\documentclass[english,reqno]{amsart}

\usepackage{float}
\usepackage[final]{graphicx}

\usepackage{graphicx}
\usepackage{xcolor}
\usepackage{color}
\usepackage{verbatim}
\usepackage{tikz}
\usepackage{soul}

\usepackage{subcaption}  

\usepackage{amsmath, appendix, ulem}
\usepackage[alphabetic,nobysame]{amsrefs}
\usepackage[nobysame]{amsrefs}
\usepackage{amssymb, color}
\usepackage[margin=1in]{geometry}

\usepackage{mathrsfs}
\usepackage{graphicx}
\usepackage{float}
\usepackage{epsf}
\usepackage{hyperref}
\usepackage{titletoc}

\graphicspath{{../Figures/}}

\numberwithin{equation}{section}

\newtheorem{thm}{Theorem}[section]

\newtheorem{question}[thm]{Question}
\newtheorem{assumption}[thm]{Assumption}
\newtheorem{prop}[thm]{Proposition}

\theoremstyle{remark}

\newtheorem*{theorem*}{Theorem}

\usepackage{amsmath}
\usepackage{amssymb}
\usepackage{amsthm}

\usepackage{bm}
\usepackage{cases}
\theoremstyle{plain}
    \newtheorem{theorem}{Theorem}[section]
    
    \newtheorem{lemma}[theorem]{Lemma}
    \newtheorem{proposition}[theorem]{Proposition}
    \newtheorem{corollary}[theorem]{Corollary}

       \newtheorem{remark}[theorem]{Remark}
\newtheorem*{remark*}{Remark}

\newcommand\numberthis{\addtocounter{equation}{1}\tag{\theequation}}

\newcommand{\half}{\frac{1}{2}}

\newcommand{\utx}{u_{(t,x)}}

\newcommand{\rd}{\mathrm{d}}

\newcommand{\RR}{\mathbb R}

\newcommand{\jhalf}{{j+1/2}}

\graphicspath{{../}}

\newcommand{\revO}[1]{\textcolor{black}{{#1}}}

\title{The Evolution of Pointwise Statistics in Hyperbolic Equations with Random Data}
\author{Alina Chertock}
\address{Department of Mathematics, North Carolina State University, Raleigh, NC 27695, USA}
\email{chertock@math.ncsu.edu}
\author{Pierre Degond}
\address{Institut de Math{\'e}matiques de Toulouse, 
CNRS \& Universit{\'e} Paul Sabatier,
31062 TOULOUSE Cedex 9, France}
\email{pierre.degond@gmail.com}
\author{Amir Sagiv}
\address{Department of Mathematical Sciences, New Jersey Institute of Technology, Newark, NJ 07102, USA}
\email{amir.sagiv@njit.edu}
\author{Li Wang}
\address{School of Mathematics, University of Minnesota, Minneapolis, MN 55455, USA}
\email{liwang@umn.edu}

\begin{document}
\maketitle
\begin{abstract}
We consider one-dimensional hyperbolic PDEs, linear and nonlinear, with random initial data. Our focus is the {\em pointwise statistics,} i.e., the 
probability measure of the solution at any fixed point in space and time. For linear hyperbolic equations, the probability density function 
(PDF) of these statistics satisfies the same linear PDE. For nonlinear hyperbolic PDEs, we derive a linear transport equation for the 
cumulative distribution function (CDF) and a nonlocal linear PDE for the PDF. Both results are valid only as long as no shocks have formed, 
a limitation which is inherent to the problem, as demonstrated by a counterexample. For systems of linear hyperbolic equations, we introduce 
the multi-point statistics and derive their evolution equations. In all of the settings we consider, the resulting PDEs for the statistics 
are of practical significance: they enable efficient evaluation of the random dynamics, without requiring an ensemble of solutions of the 
underlying PDE, and their cost is not affected by the dimension of the random parameter space. Additionally, the evolution equations for the 
statistics lead to a priori statistical error bounds for Monte Carlo methods (in particular, Kernel Density Estimators) when applied to hyperbolic PDEs 
with random data.
\end{abstract}

\section{Introduction}
In mathematical modeling, the inputs of an otherwise deterministic model are often random or uncertain.  It is then natural to study the 
{\em quantity of interest} (model output) as a random variable. This task, known as uncertainty propagation (or a forward problem in 
uncertainty quantification), is prevalent in many scientific and engineering fields.  In particular, such approaches have been
successfully applied to hyperbolic partial differential equations (PDEs) and conservation laws, see, e.g., \cite{abgrall2017uncertainty, jin2017uncertainty, mishra2012sparse, mishra2013multi, muller2013multilevel, peherstorfer2020model, pettersson2009numerical, poette2009uncertainty, tryoen2010intrusive}. Such hyperbolic equations, in turn, are fundamental models in such diverse fields as gas dynamics, shallow water 
equations, magneto-hydrodynamics in plasma physics, car and pedestrian traffic, and nonlinear elasticity.

Common computational approaches to uncertainty propagation utilize an ensemble of solutions to the underlying model, thereby approximating 
the statistical properties of the quantity of interest (see Sec.\ \ref{sec:motivation}). Moreover, even if the quantity of interest is a 
property of the solution of a PDE at some point in time, such methods will usually disregard the {\em dynamics} of that quantity in time. 

{\em In this work, we study the dynamics of measures associated with the solutions of (deterministic) hyperbolic PDEs with random initial 
data.} Concretely, consider  a family of one-dimensional hyperbolic equations with random initial conditions,
\begin{equation}\label{eq:genSet}
\begin{cases}
\partial_t u(t,x;\omega )= c(x)\partial_x \left[a(u)\right]  \, ,
\\ u(0,x,\omega) = u_{(0,x)}(\omega)  \, .
\end{cases}
\end{equation} Here  $u:[0,T]\times \RR_x \times \Omega \to \RR^n $ for some $T>0$, the parameter space $\Omega \subseteq \RR^d$ is equipped with a Borel probability measure $\mu$, for every $x\in \RR$ the initial data $u_{(0,x)}:\Omega \to \RR ^n$ is a measurable function, $c:\RR_x \to \RR^n$, and $a(u)$ is either the identity matrix or a strictly convex function (see respective sections for details).
We study the {\em pointwise statistics $\phi_{t,x}$}: a measured-valued function (a Young Measure \cite{balder1995lectures, pedregal1997parametrized}), which for every fixed $x\in \RR$ and $t\geq 0$ is defined as
\begin{subequations}\label{eq:phiDef}
\begin{equation}
\phi_{t,x} \equiv u(t,x;\cdot )_{\sharp} \mu \, ,
\end{equation}
where we recall that the {\em pushforward probability measure}  of $\mu$ by the solution operator $u(t,x;\cdot)$ satisfies
\begin{equation} \phi _{t,x}(A) = \mu \left(\left\{ \omega \in \Omega ~~ | ~~ u(t,x;\omega)\in A  \right\}\right) \, ,
\end{equation}
\end{subequations} for all Borel measurable sets $A \subseteq \RR^n$.\footnote{In the literature $\phi_{t,x}$ is also known as the ``law'' of the solution. To avoid ambiguity, we refrain from this term.}
We ask:
\begin{question}\label{q:main}
    Can one determine $(t,x)\mapsto \phi_{t,x}$, the pointwise statistics at time $t\geq 0$, solely based on their initial value $x\mapsto \phi_{0,x}$? 
\end{question}

Fundamentally, it is worth noting that for general initial value problems with random data, the answer to Question \ref{q:main} is {\em negative:} the random field $u(t,x;\cdot)$ contains information regarding the correlations between different points in space and time, whereas the pointwise statistics $\phi_{t,x}$ does not. Put differently, even full knowledge of $x\mapsto \phi_{t,x}$ tells us nothing about the {\em joint} probability of $u(t,x;\omega)$ and $u(t,y;\omega)$ for any $x\neq y$. Since the mapping between $u(t,x;\omega)$ to $\phi_{t,x}$ is not injective, knowledge of $\phi_{0,x}$ should not be sufficient to predict $\phi_{t,x}$. 

 Notwithstanding, and using the fact that $u(t,x;\omega)$ is constrained by being the solution of a hyperbolic PDE such as \eqref{eq:genSet}, one can still answer Question \ref{q:main} affirmatively. The importance of the hyperbolicity of the underlying ODE can be seen, for example, by our use of the method of characteristics (see e.g., \cite[Sec.\ 3.2]{evans2022partial}) in proving the main results. To date, we {\em do not know of any other category of PDEs} for which the answer to Question \ref{q:main} is positive.

\subsection{Main results}\label{sec:new} In the {\bf linear scalar} (possibly heterogeneous) transport equation $$\partial_t u(t,x;\omega) =c(x)\partial_x u \, ,$$
the probability density function (PDF) of $\phi_{t,x}$, denoted by $f(t,x,\cdot)$, satisfies {\em the same} PDE (Theorem \ref{thm:cxPDF}) $$\partial_t f(t,x,U) =c(x)\partial_x f \, .$$  This description provides a positive answer to Question \ref{q:main}: since $f$ satisfies a linear transport equation, the initial data $f(0,\cdot ,\cdot)$ determines completely $f(t,\cdot, \cdot)$ for all later times. We provide two new proofs for this result - one using a weak formulation of the transport equation (weak in the random variable $\omega$, not in the space variable $x$), and the other using the method of characteristics.

In the {\bf scalar nonlinear case} 
$$ \partial_t u(t,x;\omega) = \partial_x \left( a(u)\right) \, , \qquad a''>0  \, ,$$
the story becomes more complicated: in the {\em absence of shocks}, the cumulative distribution function (CDF) $F(t,x,U) \equiv \phi_{t,x}(-\infty, U)$ satisfies a family of {\em linear} equations (Theorem \ref{thm:CDF_PDE}) $$\partial_t F +a'(U)\partial_x F = 0 \, ,$$  while the PDF satisfies a nonlocal PDE (Theorem \ref{thm:PDF_PDE_nl}). However, {\em when a shock forms, the answer to Question \ref{q:main} is negative.} We show an explicit example where knowledge of $\phi_{t,x}$ before the shock formation is {\em not sufficient} to determine $\phi_{t,x}$ for later times (Sec.\ \ref{sec:ShockExample}). 

For a broad class of {\bf systems of linear hyperbolic equations,} the answer to Question~\ref{q:main} is negative~(Sec.\ \ref{sec:sys}): $\{\phi_{t,x}\}_x$ cannot be determined from $\{\phi_{0,x}\}_x$, due to the importance of pointwise correlations between the different components of the vector-valued random fields (Prop.\ \ref{prop:sys_bad}). Instead, for $N$-dimensional systems we introduce the {\em $N$-points statistics} $\phi_{t,x_1, \ldots , x_N}$, and show that it can be propagated in time, and its PDF satisfies a $N+1$-dimensional linear transport equation (Prop.\ \ref{prop:Npoints}). 

Finally, in Section \ref{sec:reg} we apply our results to the {\bf statistical} study of {\bf Monte Carlo} methods in uncertainty quantification (UQ): using the evolution equations for $\phi_{t,x}$ derived in the previous sections, we show that the PDF $f(t,x,U)$ maintains the regularity of the initial PDF $f(0,x,U)$ (with analogous results for the CDF). We use that regularity result to prove {\bf a-priori statistical convergence rates} for a popular non-parametric estimator of $f(t,x,U)$, kernel density estimators (KDE, see \cite{wasserman2006all}). Hence, even without solving the evolution equations for the PDF and CDF, their mere existence provides insight into the accuracy of the more standard Monte-Carlo methods.

\subsection{Motivation}\label{sec:motivation}

Computing $\phi_{t,x}$ is a quintessential task in the field of UQ, often known as {\em uncertainty propagation}: given a deterministic model with random parameters or inputs  (usually a PDE), one is interested in the statistical characterization of the model outputs, e.g., of the solution $u(t,x;\omega)$. While traditionally the UQ literature focuses on moment estimation \cite{ghanem2017handbook, sudret2000stochastic, xiu2010numerical}, in many applications it is necessary to obtain the ``full statistics,'' i.e., the PDF or CDF of the quantity of interest \cite{chen2005uncertainty, colombo2018basins, ditkowski2020density, le2010asynchronous, najm2009uncertainty, patwardhan2019loss, sagiv2020wasserstein, sagiv2022spectral, sagiv2017loss, zabaras2007sparse}. 
Standard approaches to uncertainty propagation can be divided into two categories: 
\begin{itemize}
    \item Monte Carlo methods. Here, i.i.d.\ samples of $u(t,x;\omega)$ are computed and used in a non-parametric density estimator \cite{tsybakov2009nonparametric, wasserman2006all}. Broadly speaking, these methods are robust and theoretically sound, but tend to converge slowly (in the number of samples, i.e., of PDE solves), and are therefore often computationally expensive.
    \item Surrogate models. Here, the function $\omega \mapsto u(t,x;\omega)$ is approximated by a simpler function, e.g., a polynomial or a spline. While constructing the surrogate model may be computationally expensive, evaluating the surrogate model is relatively inexpensive computationally. Hence, the statistics associated with the surrogate model can be obtained to arbitrary precision. The accuracy of surrogate models relies on the smoothness of the map $\omega \mapsto u(t,x;\omega)$, which can only be a-priori guaranteed for certain equations (for such examples, see e.g., \cite{cohen2010convergence, jahnke2022multilevel}). Another fundamental challenge of surrogate models is the curse of dimensionality: without special care (e.g., the design of sparse grids \cite{back2011stochastic}), the computational cost of a surrogate model with fixed accuracy would increase exponentially in the dimension of $\Omega$.
\end{itemize}   
Uncertainty Propagation in {\em hyperbolic equations} has garnered a lot of attention over the years, with hundreds of works using both Monte Carlo methods, e.g., \cite{mishra2012sparse, mishra2013multi, muller2013multilevel}, and surrogate models, e.g., \cite{pettersson2009numerical, poette2009uncertainty, tryoen2010intrusive}.
See \cite{abgrall2017uncertainty, jin2017uncertainty} and the references therein for a broader and systematic review.

 From a {\bf practical point of view,} any description of the dynamics of $\phi_{t,x}$ (i.e., a positive answer to Question \ref{q:main}) {\em circumvents} the need for a large ensemble of solutions of the underlying PDE \eqref{eq:genSet}. From this perspective, our results show that (in some cases) the pointwise statistics  $\phi_{t,x}$ can be studied directly, without resorting to the costly processes of sampling from or approximating the random field $u(t,x;\omega)$.

 \subsection{Existing literature}\label{sec:lit}

Let us first mention two other notable analogous settings in which Question \ref{q:main} has been explored. First, consider an {\em Ordinary} Differential Equations (ODE) of the form $\dot{\bf y} (t) = Q({\bf y})$ and where the initial conditions ${\bf y}(0) = {\bf y_0}\in \RR ^n$ are randomly distributed according to the probability measure~$\rho_0$. Define the solution operator (flow) for time $t\geq 0$ as $Y_t {\bf y_0} \equiv {\bf y}(t)$. Then  ${\bf y}(t)$ is distributed according to $\rho(t)\equiv (Y_t)_{\sharp} \rho_0$, which satisfies the $(n+1)$-dimensional continuity equation \cite[Chapter 4]{santambrogio2015optimal}:
\begin{equation}\label{eq:contEq}
         \partial_t \rho (t,y) + \vec \nabla_{y} \cdot \left(\rho Q\right) = 0 \, , \qquad \rho(0,\cdot )= \rho_0 \,.
     \end{equation}
     For presentation purposes, it suffices to consider an absolutely continuous $\rho (t)$ and to identify it with its PDF, $\rho(t,\cdot)\in C([0,\infty) \to L^1(\RR^n))$. \footnote{For a general measure, the derivatives need to be taken in an appropriate sense, see \cite[Chapter 4]{santambrogio2015optimal}.} We now see that $\rho(t,y)$ is not the pointwise statistics $\phi_{t,x}$. \footnote{The analog of $\phi_{t,x}$ in the ODE settings would be the coordinate-wise statistics $\phi(t,j)$ for $1\leq j \leq n$, defined as the corresponding marginal, e.g., $\phi_{t,1}(A) \equiv \rho(t)(A\times \RR^{n-1})$ for all measurable $A\subseteq \RR$. 
     The analog of $\rho(t)$ in the {\em PDE} settings can be defined as follows: $U_t u_0(\cdot) \equiv u(t,\cdot)$ is the flow/solution operator of the hyperbolic initial value problem \eqref{eq:genSet}. Then $\rho_0$ would be the measure on the function space describing the distribution of $u_{0,\cdot}(\omega)$ is, and $\rho (t) \equiv (U_t)_{\#}\rho _0$. This is still a probability measure on function spaces, as opposed to $\phi_{t,x}.$}  
     Therefore, Question \ref{q:main} and our main results are not an immediate generalization of \eqref{eq:contEq}. 

     The second analogous settings are that of {\em Stochastic} PDEs of the form
     $$ dX_t = \mu (t,X_t)\,dt + \sigma(t,X_t)dW_t \, .$$ Here, the randomness in the solution is due to the equation itself, rather than the initial data, and the pointwise statistics $\phi_{t,x}$ satisfy a deterministic PDE known as the  Fokker Planck equation (FPE). Perhaps the simplest and most well-known settings are those of driftless homogeneous white noise ($\mu \equiv 0$ and $\sigma(t,X_t)\equiv X_t$), in which the FPE reduces to the heat equation. For further details, see e.g.,  \cite{pavliotis2014stochastic, risken1996fokker}.

Turning to the present settings of {\em hyperbolic PDEs}, evolution equations for the PDF and CDF, such as \eqref{eq:PDF_utcxux} and \eqref{eqn:CDF_noshock}, were previously derived in an insightful line of works under the name ``the method of distribution''. See e.g., \cite{boso2016method, boso2020data, tartakovsky2015method ,yang2020method,  yang2022method} and the references therein. These works provide algorithms for various use cases, with a focus on practical applications. 

This work makes several substantial contributions, with the long-term goal of developing a fundamental understanding of Question \ref{q:main}. First, our methods of proof are new. In particular, our use of the method of characteristics throughout this work helps identify what is unique about hyperbolic equations that yields evolution equations for $\phi_{t,x}$. Our measure-transport perspective also provides a clear method for deriving the dynamics of one-point statistics. Second, by providing definitive negative answers to Question \ref{q:main} in the cases of shock formation and linear systems, we identify the limitations of the pointwise statistics: notably, for linear systems, we identify the $N$-points statistics $\phi _t(x_1, \ldots ,x_N)$ as the fundamental object that can be propagated in time. Finally, we establish a linkage between the evolution equations for the CDF and PDF, their regularity, and statistical error analysis of Monte Carlo methods, which we believe could be a helpful approach to studying Monte Carlo methods in UQ more broadly.

Related to this work is the concept of {\em statistical solutions} of hyperbolic conservation laws \cite{fjordholm2017statistical, fjordholm2020statistical}. The original motivation in \cite{fjordholm2017construction} is the well-posedness of systems of $N$ hyperbolic conservation laws $\partial_tu + \vec\nabla_x \cdot f(u) = 0$, in dimension $d>1$. The idea was to generalize the PDE by embedding the initial data of the PDE in the space of Young measures, and then seeking the corresponding measure-valued solution. In detail, the initial data $u_0:\RR ^d \to \RR^N$ is identified with the (atomic) Young measure $x\mapsto \delta_{u_0(x)}$, which is the initial data for a corresponding {\em measure-valued} Cauchy problem \begin{equation}\label{eq:fjordholmMV}
\partial_t\langle \nu(t,x), {\rm id}\rangle + \nabla _{x} \langle \nu(t,x), f\rangle = 0 \, ,
\end{equation}
where a measure acts on a function by $\langle \nu(t,x), g\rangle = \int\limits g(U) d\nu (t,x)(U) $. A main result of \cite{fjordholm2017construction} is that \eqref{eq:fjordholmMV} with ``deterministic'' (or atomic) initial condition admits an entropy-minimal solution (in an appropriate sense). 

In the settings of this manuscript, \eqref{eq:genSet}, solutions to \eqref{eq:fjordholmMV} with $\nu(0,x)=\phi_{0,x}$ are far from unique: for example, even for the simplest linear and scalar conservation law, $u_t+u_x = 0$,  \eqref{eq:fjordholmMV} reads as $$\partial_t \mathbb{E}_{\omega}[u(t,x;\omega)] + \partial_x  \mathbb{E}_{\omega}[u(t,x;\omega)] = 0\, ,$$
i.e., it can only predict the mean, and not the full measure $\phi_{t,x}$.

In \cite{fjordholm2017statistical}, this non-uniqueness issue is solved by embedding \eqref{eq:fjordholmMV} into an {\em infinite hierarchy} of evolution equations not only for $\phi_{t,x}$, but also to the $k$-points statistics, e.g., for $k=2$ the quantity $\phi_{t,x,y}$ defined in Section \ref{sec:sys}. On the one hand, \cite{fjordholm2017statistical} proves that this infinite hierarchy of pointwise correlations is guaranteed to have a solution in time. On the other hand, it is an infinite hierarchy (and so cannot be numerically computed), and it is indeed shown to be equivalent to the evolution of the entire random field $u(t,x;\omega)$. From both a practical and a fundamental perspective, it is interesting that this hierarchy can be truncated after a finite number of points: in this paper, $k=N$ for $N$ linear hyperbolic equations, and $k=1$ in nonlinear equations without shocks.

\subsection{Discussion and Consequences}
The positive answers to Question \ref{q:main} lead to several consequences which, to the best of our knowledge, were not previously noted.

{\bf Spatial correlations are not necessary.} The pointwise statistics $\{ \phi _{t,x} \}_{(t,x)\in \RR_t \times \RR_x}$ is a two-dimensional parametric family of probability measures on $\RR$, {\em but not a random field,}  i.e., not a probability measure of some function space from $\RR_t \times \RR_x$ to $\RR_U$. In particular, $\{\phi_{t,x}\}_{(t,x)}$ does not contain information about spatial correlations between different points in space and time. For example, even full knowledge of $\{ \phi _{t,x} \}_{(t,x)\in \RR_t \times \RR_x}$ cannot be used to compute $\mu \left\{ \omega ~~|~~ u(t,x_1;\omega) > u(t,x_2;\omega)\right\}$ for $x_1\neq x_2$ and some $t>0$.
    
    Hence, Theorem \ref{thm:cxPDF} shows that while an initial condition is expressed in terms of its dependence on $\omega$, it is only the PDF $\phi_{0,x}$ that determines $\phi_{t,x}$ later on. Therefore, rather than tracking the (potentially infinite-dimensional) random field $u(t,x;\omega)$ and {\em derive } $\phi_{t,x}$, one can simply evolve in time the $2$-dimensional scalar function $t\mapsto f(t,x;U)$.

{\bf  Canonical choice of $(\Omega, \omega)$ and dimensionality reduction.} Suppose that, in addition to \eqref{eq:genSet}, we are given another set of initial data
$$
 \tilde u(0,x,\gamma) = \tilde{u}_{(0,x)}(\gamma)  \, ,
$$
where $\gamma$ is in {\em another} probability space $(\Gamma,  \mu _{\gamma})$.
Our results then guarantee that if $$\phi_{(0,x)}= u_{0,x}(\omega)_{\sharp}\mu = \tilde{u}_{0,x}(\gamma)_{\sharp} \mu_{\gamma} \, , \qquad \forall x\in \RR\, , $$
then the corresponding PDFs $\phi_{t,x}$ will remain equal for all times.

 We can thus propose a {\em canonical} choice of $\omega$ and $\tilde{u}(0,x;\omega)$: {\em the inverse CDF.} Set $\Omega=[0,1]$ with the uniform probability (Lebesgue) measure, and $$u_{0,x}(\omega) = F_{(0,x)}^{-1}(\omega) \equiv   \inf \{ U ~~ | ~~ F(U)\geq \omega) \}\, ,$$ the left-continuous inverse of the CDF (or quantile function). It is  standard that $(F^{-1}_{(0,x)})_{\sharp} \lambda = \phi_{(0,x)}$, see e.g., \cite[Chapter 2]{santambrogio2015optimal}.

 {\em This canonical choice of initial condition serves as a dimensionality reduction technique,} which can be used as a pre-processing step for surrogate models: if all we are interested is in computing $\phi_{t,x}$ at later times, then it is sufficient to consider the one-dimensional construction above, thus allowing us always to construct surrogate models in one-dimensional settings.\footnote{We emphasize that this approach to dimension reduction will not provide information about correlations between different points in space, since the resulting random fields $u$ and $\tilde{u}$ may very well be different.}

{\bf Comparison with existing algorithms.}  For a given $t_*,U_*>0$ and  $x_*\in \RR$, computing  $f(t_*,x_*,U_*)$ can be done by {\em (i)} Computing $f(0,x,U_*)$ directly for all $x\in \RR$, and {\em (ii)} Solving the PDE \eqref{eq:PDF_utcxux} once, up to time $t_*$. An analogous procedure can be proposed for computing the CDF $F(t_*,x_*,U_*)$. 
How does this approach compare with existing algorithms for finding $f(t,x,U)$?

The accuracy of the proposed algorithm is as good as the accuracy of the PDE solver. The computational downside is that for {\em every} choice of $U_* \in \RR$, one needs to repeat the PDE solve anew.

In comparison, a {\bf Monte Carlo method,} would draw $N\geq 1$ i.i.d.\ samples $\omega_1, \ldots, \omega_N \sim \mu$, and for each of which solve the original PDE \eqref{eq:linScalar} to obtain an approximation of $\{u(t,x;\omega_j)\}_{j=1}^N$. Here the complexity is $N$-times that of the PDE solver, but we immediately get an estimate for $f(t,x,U)$ for {\em all} $(t,x,U)$. The accuracy is dominated by two factors: the statistical error of the MC method, which {\em at best} scales at the parametric rate $\mathcal{O}(N^{-1/2})$; and the numerical error of the PDE scheme. 

Similarly, {\bf surrogate models} (e.g., collocation gPC), use $N$ solutions of the original PDE \eqref{eq:linScalar} to obtain $\{u(t,x;\omega_j)\}_{j=1}^N$, and thereby obtain an approximation of $f(t,x,U)$ for all $x$ and $U$. As noted above, the accuracy of such methods is preferable to that of Monte Carlo methods in moderate dimensions, but deteriorates exponentially with the dimension of $\Omega$ without any special treatment \cite{ghanem2017handbook, sudret2000stochastic}.

{\bf Regularity Theory and Statistical Error Estimates for Monte Carlo Methods.} In those cases where the PDF $f(t,x,U)$ or CDF $F(t,x,U)$ are governed by a PDE, we can deduce from it their regularity for all time $t>0$. Furthermore, regularity theory could yield an a priori statistical error analysis for Monte Carlo methods when applied to these equations. See Sec.\ \ref{sec:reg} for details.

\subsection{Open Questions}

\begin{enumerate}

     \item Theorem \ref{thm:CDF_PDE}, which concerns the  CDF associated with nonlinear scalar conservation laws, is only valid outside of the ``shock envelope'', the regions in the $(t,x)$ plane where the solution exists globally with probability $1$. This limitation is demonstrated to be a fundamental one by a numerical example (Sec.\ \ref{sec:ShockExample}).     One intriguing approach to overcoming this limitation is that of \cite{boso2020data}, where a PDE for the CDF is deduced even in the presence of shocks; however, it contains an unknown source term that the authors learn from data.  In \cite{fjordholm2017statistical}, this issue is bypassed rigorously by considering the {\em infinite} hierarchy of multi-point statistics $\{\phi_{t,x_1,\ldots,x_k}\}_{k\geq 1}$. Since these infinite hierarchies cannot be computed in full, the question of their computational approximation is then studied in \cite{fjordholm2020statistical}.  Whether or not some knowledge of the shocks can be incorporated into a {\em finite} statistical object whose evolution can be described in closed form, remains an open question.

    \item Theorem \ref{thm:2sys} hinges on Assumption \ref{as:P_x0}: for linear systems of evolution equations, the speed of propagation matrix, $D(x)$ in \eqref{eq:2sys_gen}, can be diagonalized by an $x$-independent matrix. The source of the issue appears to be technical and is discussed in Remark \ref{rem:2sysPx}. Whether an evolution equation for $\phi_{t,x,y}$ associated with the general system \eqref{eq:2sys_gen} {\em without limiting assumptions} remains open.
    
    \item Scalar nonlinear equations $u_t + \partial_x a(u) = 0$ with $a'' >1$ can be rewritten in relaxation form: the limit $\varepsilon \to 0$ of the system (after a change of coordinates) \cite{jin1995relaxation}:
\begin{equation} \label{eqn:cl-relax}
\begin{cases}{}
p_t + p_x = - \frac{1}{\varepsilon} \left( \frac{p- q}{2} - a\left( \frac{p + q}{2}\right) \right) \equiv - b(p, q) \, ,
\\ q_t - q_x = b(p, q) \, ,
\end{cases}
\end{equation}
where we solve for unknown $p(t,x)$ and $q(t,x)$. This seemingly alternative approach to propagating $\phi_{t,x}$ in the presence of shocks, however, proves to be challenging. As noted in the previous open problem, even for a fixed $\varepsilon>0$, two-point statistics associated with \eqref{eqn:cl-relax}, which we denote by $\phi_{t,x,y}^{\varepsilon}$, cannot be obtained from Theorem \ref{thm:2sys} -- due to the mixed sources, i.e., the $p$ and $q$ dependence of the right-hand side in \eqref{eqn:cl-relax}. If one was able to propagate in time $\phi_{t,x,y}^{\varepsilon}$, then the challenge would be to understand the limit $\lim_{\varepsilon\to 0}\phi_{t,x,y}^{\varepsilon}$ and see whether it is related to the original nonlinear problem.
    
\end{enumerate}

\section{Scalar linear transport equation}\label{sec:scalar}
We begin with linear, scalar, variable-coefficient (heterogeneous) hyperbolic equations. Though the results here could be derived as a special case of the nonlinear case (Sec.\ \ref{sec:nonlin}), we consider these separately for presentation purposes: the proofs here are simpler, and there are almost no technical caveats. Concretely, consider $u:\RR_t \times \RR_x \times \Omega \to \RR$ satisfying\footnote{A sufficient regularity condition for existence and uniqueness is that  $u_{0,x}(\omega)$ is piecewise $C^1$ in $x$ and $c\in C^1 (\RR)$. We do not pursue here the question of minimal regularity, such as in \cite[Chapter 7]{evans2022partial}.}
\begin{equation}\label{eq:linScalar}
\begin{cases}
\partial_t u(t,x;\omega )= c(x)\partial_x u  \, ,
\\ u(0,x,\omega) = u_{(0,x)}(\omega)  \, ,
\end{cases}
\end{equation}
for some $c\in C^1(\RR)$. 
\begin{theorem}\label{thm:cxPDF}
Consider \eqref{eq:linScalar} with $c\in C^1 (\RR)$.
Assume that $f_0(x,U)$ is piecewise $C^1$ in $x$ for every $U\in \RR$.  Then  $f(t,x,U)$, the PDF of the pointwise statistics $\phi_{t,x}$ satisfies the following equation 
\begin{equation}\label{eq:PDF_utcxux}
\begin{cases}
\partial_t f(t,x;U )= c(x)\partial_x f(t,x;U) \, ,
\\ f(0,x,U) = f_0(x,U)  \, ,
\end{cases}
\end{equation}
where $f_0(x,U)$ is the PDF of $\phi_{0,x}$.
\end{theorem}

While a version of this result appeared in \cite{tartakovsky2015method}, we provide two new proofs in Sec.\ \ref{sec:cxPDF_pfs}.

 \begin{remark*}[Explicit example]   Consider $\Omega = [0,1]$ with the uniform Lebesgue measure, and consider $u_t=u_x$ with initial data $u_{(0,x)} (\omega) = (1+\omega) e^{-x^2}$.
 At $t=0$, the measure $\phi(0,x)$ is the uniform probability measure on $I_0(x) \equiv [e^{-x^2}, 2e^{-x^2}]$, for every $x\in \RR$. Therefore, $f(0,x;U)= |I_0(x)|^{-1}\chi_{I_0(x)}(U)$, where we denote by $\chi_{[a,b]}$ the identifier function of an interval $[a,b]$. Since the solutions are traveling wave $u(t,x;\omega) = u(0,x+t;\omega)$, we have that $f(t,x,U)= |I_0 (x+t)|^{-1} \chi_{I_0 (x)} (U)$. Indeed, $f(t,x,U)$ satisfies \eqref{eq:PDF_utcxux} as expected by Theorem \ref{thm:cxPDF}.   
\end{remark*}


\subsection{Two proofs of Theorem \ref{thm:cxPDF}}\label{sec:cxPDF_pfs}
\begin{proof}[Proof of Theorem \ref{thm:cxPDF}]
    Fix $t,x\in \RR$. Since $\phi_{t,x}$ is the pushforward measure of $\mu$ by the evolution operator $u(t,x;\cdot)$, then for every $h~\in~C_c^{\infty}(\RR)$ we have that 
    \begin{align} \label{push}
       \int\limits_{\RR} f(t,x,U) h(U) \, \rd U = \int\limits_{\Omega} h(u(t,x;\omega)) \, \rd \mu (\omega) \, . 
    \end{align}
    Differentiating both sides of \eqref{push}, we get on the left-hand side 
    $$  \partial_t \int\limits_{\RR} f(t,x,U) h(U) \, \rd U = \int\limits_{\RR} \partial_t f(t,x,U) h(U) \, \rd U \, , $$
    and on the right-hand side, 
    \begin{align*}
         \partial_t \int\limits_{\Omega} h(u(t,x;\omega)) \, \rd \mu (\omega)  
        & = \int\limits_{\Omega} h'(u(t,x;\omega)) \partial_t u(t,x;\omega) \, \rd \mu (\omega)  \\
        [{\rm using}~\eqref{eq:linScalar}] \qquad \qquad  \qquad & = \int\limits_{\Omega} h'(u(t,x;\omega)) c(x) \partial_x u(t,x;\omega) \, \rd \mu (\omega)  \\
        & = c(x)\int\limits_{\Omega} h'(u(t,x;\omega))  \partial_x u(t,x;\omega) \, \rd \mu (\omega) \\
        & = c(x)\partial_x \int\limits_{\Omega} h(u(t,x;\omega))  \, \rd \mu (\omega)  \\
        [{\rm using}~\eqref{push}] \qquad \qquad \qquad & =c(x) \partial_x \int\limits_{\RR} f(t,x,U) h(U) \, \rd U \\
        & = \int\limits_{\RR}  c(x) \partial_x f(t,x,U) h(U) \, \rd U \, .
    \end{align*}
    Hence, 
    \begin{equation}
        \int\limits_{\RR} \partial_t f(t,x,U) h(U) \, \rd U = \int\limits_{\RR}  c(x) \partial_x f(t,x,U) h(U) \, \rd U \, ,
    \end{equation}
    for every $h\in C_c^{\infty}$, and therefore by density, for every $h\in L^1$. Hence, the integrands are equal, and we have the PDE \eqref{eq:PDF_utcxux}.

Note that, since the initial condition is piecewise $C^1$ in the $x$-variable, the PDE \eqref{eq:PDF_utcxux} itself guarantees that $f(t,\cdot ;\cdot)$ remains in the same space for all $t\geq 0$, and hence the derivation is valid.
\end{proof}

The alternative proof to Theorem \ref{thm:cxPDF} uses the method of characteristics:
\begin{proof}[An alternative proof for Theorem \ref{thm:cxPDF}.]
    Fix $\omega\in \Omega$ such that $x\mapsto u_{0,x}(\omega)$ is in $C^1$. Then the solution to \eqref{eq:linScalar} can be written, by the method of characteristics, as 
    $$u(t,x;\omega) = u_{0}(\xi_x (t); \omega) \, ,  $$
    where $\xi_x (t)$ solves the ODE 
    \begin{equation}\label{eq:cx_char}
\begin{cases}
\dot{\xi}_x (t) = -c(\xi_x) \, ,
\\ \xi_x (0) = x \, .
\end{cases}
\end{equation}
Fix $t\geq 0$ and define $q _t (x)$ to be the map $x\mapsto \xi_x (t)$. From ODE theory for existence and uniqueness, $q_t$ is injective, and therefore $u(t,x;\omega) =u_{0} (q_t ^{-1}(x);w)$. Hence, for any measurable $A\subseteq \RR$
\begin{align*}
     \phi_{t,x}(A) &= \mu \left( \{ \omega \in \Omega ~~| ~~ u(t,x;\omega )\in A \}\right) \\ 
     &= \mu \left( \{ \omega \in \Omega ~~| ~~ u_{0}( q^{-1}_t (x);\omega) \in A \}\right) = \phi_{t,q_t^{-1} (x)} (A) \, ,
\end{align*}
and therefore the PDF satisfies $f(t,x, U) = f_0 (q_t^{-1}(x),U)$ for all $t\in \RR$.
\end{proof}

\subsection{Extension to randomness in coefficient}
We now extend the initial randomness in Equation \eqref{eq:linScalar} to include randomness in the coefficient as well. While similar results have been obtained in \cite{tartakovsky2015method}, in this brief section, we show how our methods provide an alternative derivation. We consider the following stochastic hyperbolic equation: 
\begin{equation}\label{eq:lin2}
\begin{cases}
\partial_t u(t,x;\omega )= c(\omega)\partial_x u  \, ,
\\ u(0,x,\omega) = u_{(0,x)}(\omega)  \, .
\end{cases}
\end{equation}
Let 
\begin{align} \label{fc}
    \phi_{t,x} := (u(t,x; \cdot), c(\cdot))_\sharp \mu\,,
\end{align}
and if $\phi_{t,x}$ is absolutely continuous, let $f(t,x,U,C)$ be its PDF. Then for every $h \in C_c^\infty(\RR \times \RR)$,
\begin{align*}
    \int_{\RR \times \RR} f(t,x; U, C) h(U,C) \rd U \rd C = \int_\Omega h(u(t,x;\omega), c(\omega)) \rd \mu (\omega)\,.
\end{align*}
Relation \eqref{fc} can be viewed as an extension of \eqref{eq:phiDef}. Then the following can be proven in an analogous way to the proof of Theorem \ref{thm:cxPDF}:
\begin{corollary}
    Consider \eqref{eq:lin2}, and assume that $x\mapsto f_0(x,U,C)$ is piecewise $C^1$ for all $U,C \in \RR$. Then the PDF $f(t,x; U, C)$ satisfies the following linear transport equation 
\begin{equation}
\begin{cases}
\partial_t f(t,x;U, C )= C\partial_x f(t,x;U,C) \, ,
\\ f(0,x,U, C) = f_0(x,U, C)  \, .
\end{cases}
\end{equation}
\end{corollary}


\section{Nonlinear equations}\label{sec:nonlin}

Consider a nonlinear scalar conservation law with random initial conditions,
\begin{equation} \label{eqn:cl}
\begin{cases}{}
u_t + a(u)_x = 0, 
\\ u(0,x,\omega) = u_{(0,x)}(\omega)\,,
\end{cases}
\end{equation}
where $a:\RR \to \RR_+$ is strictly convex. We define $\phi_{t,x}$ as in \eqref{eq:phiDef}, and as before define the corresponding CDF as $F(t,x,U) \equiv \phi_{t,x} \left( (-\infty, U] \right)$ and the corresponding PDF as $f(t,x,U) =\partial_U F(t,x,U) $.  Throughout, we assume that up to time $T\geq 0$, no shocks have formed. This assumption allows for the methods of characteristics to be used to derive a linear hyperbolic evolution equation for the CDF (Theorem \ref{thm:CDF_PDE}). While it is possible to differentiate the resulting CDF $F(t,x,U)$ to obtain $f(t,x,U)$, we also provide a direct derivation of the non-local linear equation that governs the propagation of the PDF (Theorem \ref{thm:PDF_PDE_nl}). Both methods of proof fall short of dealing with shock formation (Remark \ref{rem:noShock}), and thus we complement our theoretical results by a numerical example which demonstrates why there {\em could be no propagator for $\phi_{t,x}$} in the presence of shocks~(Sec.\ \ref{sec:ShockExample}).

\subsection{CDF evolution in the absence of shocks}
\begin{theorem}\label{thm:CDF_PDE}
    Consider \eqref{eqn:cl}, and assume that $\mu$-almost everywhere {\it (i)} $u_{(0,x)}(\omega) > 0$, and {\it (ii)} there exists $T\leq \infty$ such that a unique strong solution of \eqref{eqn:cl} exists for time $t\in [0,T]$. Then for all such $t$,
    \begin{equation} \label{eqn:CDF_noshock}
\begin{cases}{}
\partial_t F(t,x,U) + a'(U)\partial _x F = 0, 
\\ F(0,x,U) = F_{(0,x)} (U)\, .
\end{cases}
\end{equation}
Equivalently, the evolution of the CDF obeys $F(t,x,U)= F(0,x-a'(U)t,U)$.
\end{theorem}

\begin{proof}
Define the {\em sublevel sets} of $u(t,x;\omega)$ by 
    $$ S(t,x,U) \equiv \left\{\omega \in \Omega ~| u(t,x;\omega) < U\right\} \, ,$$ for every $U\in \RR$. We can then rewrite the CDF as follows
    $$F(t,x,U) = \int\limits_{S(t,x,U)}   \rd \mu(\omega) \, . $$
    It is therefore sufficient to show that 
    \begin{equation}\label{eq:sublevels_target}
        S(t,x,U) = S(0,x-a'(U)t,U) \, .
    \end{equation}

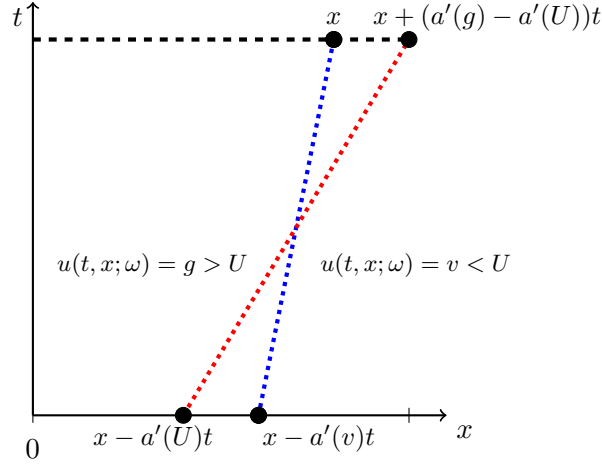
\begin{figure}[h]
    \centering
    \begin{tikzpicture}
        \draw[->, thick] (0,0) -- (5.5,0) node[below right] {\large $x$};
        
        \draw (5,0.1) -- (5,-0.1) node[below] {};
        
        \draw[->, thick] (0,0) -- (0,5.5) node[left] at (0,5.3) {\large $t$};
        
        \foreach \x in {0} {
            \draw (\x,0.1) -- (\x,-0.1) node[below] at (\x,-0.2) {\large $\x$};
        }
        
        \draw[dashed, ultra thick] (0,5) -- (5,5);
        
        \filldraw[black] (4,5) circle (3pt);
        
        \draw[blue, ultra thick, dotted] (4,5) -- (3,0);
        
        \filldraw[black] (3,0) circle (3pt);
        
        \node[anchor=west] at (3.7,2) {\color{black}\small $u(t,x;\omega)=v <U$};
        
        \draw[red, ultra thick, dotted] (2,0) -- (5,5);
        
        \filldraw[black] (2,0) circle (3pt) node[below, xshift =-4mm] {$x-a'(U)t$};
        \filldraw[black] (3,0) circle (3pt) node[below, xshift =8mm] {$x-a'(v)t$};
 \filldraw[black] (4,5) circle (3pt) node[above, yshift=0.8mm] {$x$};
\filldraw[black] (5,5) circle (3pt) node[above right, xshift=-6mm] {$x+(a'(g)-a'(U))t$};

        \node[anchor=west] at (0.2,2) {\color{black}\small $u(t,x;\omega)=g>U$};
    \end{tikzpicture}
    \caption{Illustration of the contradiction in the proof of Theorem \ref{thm:CDF_PDE}. Tracing the line from $(x,t)$ to $(x-a'(v)t,0)$ (dotted blue), and using the fact that $a'(U)>a'(v)$ due to convexity, we see that if $u(0,x-a'(U)t)=g>U$, it would violate the principle that characteristic lines do not intersect (dotted red line).}
    \label{fig:cartoonStxu}
\end{figure}

    First, take any $\omega \in S(t,x,U)$. Then, by definition, there exists $v\in [0,U)$ such that $u(t,x;\omega)=v$. By the method of characteristics, $u(0,x-a'(v)t;\omega) = v$. Now, due to the strict convexity of $a$, then $a'(U)> a'(v)$, and so $x-a'(U)t < x-a'(v)t$. Suppose that $\omega \notin S(0,x-a'(U)t,U)$, this would mean that $u(0,x-a'(U)t;\omega) =g >U$. But then $u(t,x+(a'(g)-a'(U))t;\omega) =g$, but that would mean that characteristic lines have crossed {\em before} the time $t$, contrary to the assumption that no shocks have formed. See Fig.\ \ref{fig:cartoonStxu}. Therefore $\omega \in S(0,x-a'(U)t,U)$, and we have proven that $S(t,x,U) \subseteq S(0,x-a'(U)t,U)$.

    To prove inclusion in the opposite direction, fix any $\omega \in S(0,x-a'(U)t,U)$, i.e., $u(0,x-a'(U)t;\omega) =v < U$. Define $x_* \equiv x+(a'(v)-a'(U))t$. By the method of characteristics, $u(t,x_* ;\omega)=v$. Since $a'$ is strictly increasing, $x_* < x$. On the other hand, if $u(t,x;\omega) = V>U$, but then $x-a'(V)t<x-a'(U)t$, and we would have that the characteristics cross, contradicting the assumption of no shocks being formed. Hence we have $S(t,x,U) \supseteq S(0,x-a'(U)t,U)$, and \eqref{eq:sublevels_target} is proven.
\end{proof}

\subsection{\revO{PDF evolution using a nonlocal PDE}}
Here we will derive an equation for the PDF $f(t,x,U)$ directly from the conservation law \eqref{eqn:cl}. 
\begin{theorem}\label{thm:PDF_PDE_nl}
Consider the scalar conservation law with random initial condition, \eqref{eqn:cl}. Assume that the PDF $f(t,x,U)$ exists for all $t\geq 0$ and $x\in \RR $, and \revO{and that the probability of shocks forming on $[0,t]$ is zero.} Then $f(t,x,U)$ satisfies
    \begin{equation}\label{eqn:PDF_PDE_nl}
\partial_t f + \frac{\partial}{\partial U} \left(a'(U) \int^U_{-\infty} \partial_x f  \rd v\right) = 0\,.
\end{equation}
\end{theorem}
By integrating \eqref{eqn:PDF_PDE_nl} with respect to $U$, one can readily obtain \eqref{eqn:CDF_noshock}, the evolution equation for the CDF. On the one hand, it is a slightly weaker result, since it {\em assumes} the existence of the PDF, i.e., that $\phi_{(0,x)}$ is absolutely continuous for all $x\in \RR$. However, we present it separately because the method of proof here is substantially different from that of Theorem \ref{thm:CDF_PDE}.
\begin{proof}
\begin{eqnarray*} 
\partial_t \int_{-\infty}^\infty f(t,x,U) h(U) \rd U &=& \partial_t \int_\Omega  h(\utx (\omega)) \rd \mu(\omega)  \nonumber
\\ &=& - \int_\Omega  h'(\utx) a'(\utx(\omega)) \partial_x \utx(\omega) \rd \mu(\omega)\,.
\label{eqn:4}
\end{eqnarray*}
Now let 
\begin{equation} \label{eqn:entropy}
\psi'(u) = a'(u) h'(u)\,,
\end{equation}
then (\ref{eqn:4}) becomes
\begin{eqnarray*}
 \partial_t \int_{-\infty}^\infty f(t,x,U) h(U) \rd U  &=& - \int_\Omega  \psi'(\utx(\omega)) \partial_x \utx (\omega) \rd \mu(\omega)  \nonumber
\\ &=& -\partial_x \int_\Omega  \psi(\utx(\omega)) \rd \mu(\omega) \nonumber \\
&=&  - \partial_x \int_{-\infty}^\infty  f(t,x,U) \psi(U) \rd U  \nonumber
\\ &=& - \partial_x  \int_{-\infty}^\infty  f(t,x,U) \left( \int^U_{-\infty} a'(v) h'(v) \rd v \right) \rd U \nonumber
\\ & = & -\partial_x \int_{-\infty}^\infty \psi'(v) \int_{v}^\infty  f(t,x,U) \rd U \rd v \nonumber
\\ & = &  -\partial_x \int_{-\infty}^\infty \psi'(v) \left( 1 - \int_{-\infty}^v f(t,x,U) \rd U \right) \rd v \nonumber
\\ & = & \partial_x \int_{-\infty}^\infty h'(v) a'(v) \int_{-\infty}^v f(t,x,U) \rd U \rd v  \nonumber
\\ & = & - \partial_x \int_{-\infty }^\infty h(v) \frac{\rd}{\rd v} \left( a'(v) \int_{-\infty}^v f(t,x,U) \rd U \right)  \rd v \,. \label{eq:new3_weak}
\end{eqnarray*}
Therefore, $f(t,x,U)$ solves \eqref{eqn:PDF_PDE_nl}.
\end{proof}

\begin{remark}[Why do we assume that there are no shocks?]\label{rem:noShock}
Our analysis of nonlinear hyperbolic equations with random data, \eqref{eqn:cl}, is valid only up to times where no shocks have formed (with nonzero probability). We claim that this is an inherent limitation of the problem, rather than the result of a technical difficulty.

First, let us examine why shock is detrimental to our methods of proof. In Theorem \ref{thm:CDF_PDE}, the method of characteristics is used to trace the evolution of the sublevel set $S(t,x,U)$ in time. When characteristic lines cross and a shock forms, such argument can no longer hold: even if $u(0,x-a'(U)t;\omega)=U$, whether or not $u(t,x;\omega)=U$ depends on where (in $(t,x)$) the shock is, which in turn requires full knowledge of  $x\mapsto u(0,x;\omega)$. But the latter information is not encoded in $\phi_{0,x}$. The importance of multi-point (in $x$) correlations is reminiscent of the work of \cite{fjordholm2017statistical}, where (for systems of nonlinear hyperbolic conservation laws) a hierarchy of $k$-points statistics for all $k\geq 1$ is needed to determine the statistical solution. In the proof of Theorem \ref{thm:PDF_PDE_nl}, already in \eqref{eqn:4} we require $\partial_x u(t,x;\omega)$ to exist in a strong sense, an assumption that fails in the presence of a shock.\footnote{A {\em weak form} of that derivation does not seem to advance us any further, unless the shock is known to form almost surely at a fixed location. We find these assumptions too stringent and therefore do not pursue this direction here.} 

Beyond that, {\em the no-shock condition is fundamental to Question \ref{q:main}}. Our {\em numerical results} in Sec.\ \ref{sec:ShockExample} demonstrate that two distinct initial random fields $u(0,x;\omega)$ and $v(0,x;\omega)$ may have the same initial pointwise statistics (i.e., $\phi^v_{0,x}=\phi^u_{0,x}$ for all $x\in \RR$), but as shocks form at different points in $(t,x,\omega)$ space for $u$ and $v$, the pointwise statistics differ as well. Therefore, in general there can be no operator or algorithm which takes the measure-valued function (or Young measure) $x\mapsto \phi_{0,x}$ and returns $x\mapsto \phi_{T,x}$, if shocks form with non-zero probability for some time $t\in [0,T)$.
\end{remark}


\section{Linear Hyperbolic systems}\label{sec:sys}

Consider a linear $2\times 2$ hyperbolic system in one space variable:
\begin{equation}\label{eq:2sys_gen}
\begin{cases}{}
\vec z_t (t,x;\omega) + D(x) \vec z_x = 0 \, ,
\\  \vec z(0,x,\omega) =  \vec z_{(0,x)}(\omega)  \, , 
\end{cases}
\end{equation}
where $x\mapsto D(x)$ is a smooth map from $\RR $ to the space of $2\times 2$ matrices, and each $D(x)$ is real diagonalizable.
\subsection{The Failure of one-point statistics}\label{sec:failure_systems}
In these settings, one can define the pointwise statistics $\phi_{t,x}$ in an analogous way to the scalar case 
$$ \phi_{t,x}(A) \equiv \mu \left(\right\{ \omega ~~ | ~~ \vec z (t,x;\omega) \in A \left\} \right) \, ,$$
for every measurable $A \subseteq \RR ^2$. Unfortunately, $\phi_{t,x}$ is in general {\em not uniquely determined by the initial pointwise statistic } $x\mapsto \phi_{0,x}$. Hence, there can be no hope to find the time-propagator of $\phi_{t,x}$ (or its density and CDF).  We show this using the following explicit counter-example:
\begin{prop}\label{prop:sys_bad}
    In a vector-valued systems such as \eqref{eq:2sys_gen}, the initial joint pointwise statistics $\{\phi_{0,x}\}_{x\in \RR}$ are not sufficient to determine $\{\phi_{t,x}\}_{x\in \RR}$ for $t>0$.
\end{prop}
\begin{proof}
Consider the one-dimensional wave equation in system form:
$$\begin{cases}{}
u_t (t,x;\omega) + u_x = 0 \, ,
\\
v_t (t,x;\omega) -v_x = 0 \, ,\\
 u(0,x,\omega) =  u_{(0,x)}(\omega)  \, , \\
 v(0,x,\omega) =  v_{(0,x)}(\omega)  \, .
\end{cases}$$
To show that the initial joint statistics $\phi_{0,x}$ do not determine the $\phi_{t,x}$ at later times, it is enough to provide two distinct sets of random initial conditions 
for which the initial joint pointwise statistics are the same, but the joint pointwise statistics at $t=1$ are different.

\begin{table}
    \centering
    \begin{tabular}{c||c|c|}
         $\omega$& $0$ & $1$ \\ \hline \hline
         $u^{(1)}(0,1;\omega)$ & $1$   & $0$\\
         $u^{(1)}(0,-1;\omega)$ & $0$ &  $0$\\
         $v^{(1)}(0,1;\omega)$ & $0$ & $0$\\
         $v^{(1)}(0,-1;\omega)$ & $1$ &  $0$\\
         \hline \hline   $u^{(2)}(0,1;\omega)$ & $1$ & $0$\\
         $u^{(2)}(0,-1;\omega)$ & $0$ & $0$\\
         $v^{(2)}(0,1;\omega)$ & $0$ & $0$\\
         $v^{(2)}(0,-1;\omega)$ & $0$ & $1$\\
    \end{tabular}
    \caption{Two different sets of random initial conditions, $(u^{(1)}, v^{(1)})$ and $(u^{(2)}, v^{(2)})$, prescribed at $x=\pm 1$ for all possible values of $\omega$.}
    \label{tab:ic_sys_ce}
\end{table}

Take $\Omega = \{0,1\}$ with the uniform measure ($\mu(0)= \mu(1)= 1/2$ ). We define the initial conditions,  $(u^{(1)}, v^{(1)})$ and $(u^{(2)}, v^{(2)})$, in Table \ref{tab:ic_sys_ce}. Note that we only specify the initial conditions at $x=\pm 1$. This is because these are the only points we need to arrive at a contradiction, and any value they assume at $x\neq \pm 1$ would not change the argument.

The corresponding pointwise measures can be directly computed to show that
$$\phi^{(1)}_{0,\pm 1} = \phi^{(2)}_{0,\pm 1} \, ,$$
as both measures assign a probability $1/2$ to a each of the two point $(1,0),(0,0) \in \RR^2$ at $x=1$, and probability $1/2$  to $(0,1),(0,0)\in \RR^2$  at $x=-1$. Hence, we established that the initial pointwise statistics of both sets of initial conditions are the same.

Now observe that the solutions to both equations are simple traveling waves in the opposite directions, i.e., $u(t,x;\omega) = u(0,x-t;\omega)$ and $v(t,x;\omega)=v(0,x+t;\omega)$ for all $t\geq 0$, $x\in \RR$ and $\omega \in \Omega$. Hence $$u^{(1)}(1,0;0)=v^{(1)}(1,0;0) = 1\, , \qquad  u^{(1)}(1,0;1)=v^{(1)}(1,0;1) = 0 \, , $$ and so $\phi_{1,0}^{(1)}(1,1)= \phi_{1,0}^{(1)} = \frac12 $. 
At the same time, the analogous calculation shows that $\phi_{1,0}^{(2)}(1,1) = 0$, which shows that $\phi^{(1)}_{1,0} \neq \phi^{(2)}_{1, 0} $,
which concludes our proof.

\end{proof}
\subsection{Two-point statistics}\label{sec:2pt}
The counterexample above demonstrates the issue with propagating $\phi_{t,x}$, but suggests the appropriate fix: in vector-valued systems, two-point correlations (in space) are significant. This motivates the following definition: denote $\vec z= (z_1, z_2)^\top$, and by a small abuse of notation, define the two-point statistics as
\begin{equation}\label{eq:2ptPhi}
    \phi_{t,x,y}(A) \equiv \mu \left( \left\{  \omega \in \Omega ~~ | ~~ \begin{pmatrix}
        z_1 (t,x;\omega) \\ z_2 (t,y;\omega)
    \end{pmatrix} \in A \right\} \right) \, , 
\end{equation}
for all measurable $A\subseteq \RR^2$. Alternatively, we can define the measure $\phi_{t,x,y}$ as a functional acting on test functions $\psi : \RR^2 \to \RR$
$$ \phi_{t,x,y} [\psi] = \int\limits_{\Omega} \psi\left(z_1(t,x;\omega), z_2(t,y;\omega) \right) \, \rd \mu(\omega) \, .$$
In what follows, we make the following assumption:
\begin{assumption}\label{as:P_x0}
The matrix $D(x)$ in \eqref{eq:2sys_gen} admits an $x$-independent diagonalization, i.e., there exists an invertible matrix $P$ and two functions, $c_1(x, u)$ and $c_2(x,v)$, such that for all $x \in \RR$
$$PD(x)P^{-1}=\begin{pmatrix}
    c_1(x) & 0 \\ 0 & c_2 (x)
\end{pmatrix} \, .  $$
\end{assumption}
\begin{remark}
   Consider the diagonalized system 
\begin{equation}\label{eq:2sys}
\begin{cases}{}
u_t + c_1(x) u_x = 0,
\\ v_t + c_2(x) v_x = 0,
\\ u(0,x,\omega) = u_{(0,x)}(\omega), \quad 
v(0,x,\omega) = v_{(0,x)}(\omega) \,,
\end{cases}
\end{equation}
and suppose that we are given the two-point statistics {\em for the new coordinates $u$ and $v$,} which for the sake of this remark we shall denote by $\tilde{\phi}_{t,x,y}$. The one-point statistics $\phi_{t,x}$ of $\vec z$ can be obtained, for any $A\subseteq \mathbb{R}^2$, using the change of variables 
\begin{align*}
  \phi_{t,x}(A) &\equiv \mu \left\{ \vec z (t,x;\omega)\in A \right\}  \\
  &= \mu \left\{ (
      u(t,x;\omega), v(t,x;\omega)
  )^{\top}\in P^{-1} A \right\} \\
  &=\tilde{\phi}_{t,x,x}(P^{-1}(A)) \, ,
\end{align*}
where $P^{-1}(A)=\{P^{-1}\vec \zeta ~~ |~~\vec \zeta \in A \}$.

\end{remark}

\noindent
Hence, under Assumption \ref{as:P_x0}, we need only to study $\phi_{t,x,y}$ for the decoupled system~\eqref{eq:2sys}.
\begin{theorem}\label{thm:2sys}
   Consider \eqref{eq:2sys}  with Assumption \ref{as:P_x0}). Let $\phi_{t,x,y}$ be the two-points statistics, defined for every $A\subseteq \RR^2$, 
   $$ \phi_{t,x,y} (A) \equiv \mu \left(\{\omega \in \Omega~~|~~ (u(t,x,;\omega), v(t,y;\omega))^{\top}\in A \}\right) \, , $$
   and $f(t,x,y,U,V)$ its PDF. If $f(0,x,y,\cdot ,\cdot)= f_0(x,y,\cdot ,\cdot ) \in L^1 (\RR^2)$, then
       \begin{equation*}
               \partial_t f(t,x,y,U,V) + c_1 (x) \partial_x f + c_2 (y)\partial _y f = 0 \, ,
       \end{equation*}
       for all $t\geq 0$ and $(x,y)\in \RR^2$, with the appropriate initial conditions.
\end{theorem}
\begin{proof}

Let $\psi:\RR^2 \to \RR$ be a test function, we write 
\begin{align*}
    \partial_t \iint\limits_{\RR ^2} &f(t,x,y,U,V) \psi (U,V) \, \rd U \rd V = \partial_t \int\limits_{\Omega} \psi \left(u(t,x;\omega), v(t,y;\omega) \right)  \, \rd \mu (\omega) \\
    &= \int\limits_{\Omega} \Big[ \partial_U \psi \left(u(t,x;\omega), v(t,y;\omega) \right)\partial_t u(t,x;\omega) + \partial_V \psi \left(u(t,x;\omega), v(t,y;\omega) \right)\partial_t v(t,y;\omega) \Big]\, \rd \mu (\omega) \\
    &= \int\limits_{\Omega} \Big[ \partial_U \psi \left(u(t,x;\omega), v(t,y;\omega) \right) \left(-c_1 (x) \partial_x u(t,x;\omega) \right) + \partial_V \psi \left(u(t,x;\omega), v(t,y;\omega) \right) \left( -c_2(y)\partial_y  v(t,y;\omega)  \right) \Big]\, \rd \mu (\omega)  \numberthis \label{eq:expanding_2pt} \\   
    &= -c_1 (x) \partial_x \int\limits_{\Omega} \psi \left( u(t,x;\omega) , v(t,y;\omega) \right) \, \rd \mu(\omega) -c_2 (y) \partial_y \int\limits_{\Omega} \psi \left( u(t,x;\omega) , v(t,y;\omega) \right) \, \rd \mu(\omega)   \\
    &= 
    \iint\limits_{\RR ^2} \left( -c_1 (x) \partial_x -c_2(y) \partial _y \right) f(t,x,y,U,V)  \psi (U,V) \, \rd U \, \rd V \, .
\end{align*}
Since this is true for all test functions, this completes the proof.

\end{proof}

What can be said about the general case, i.e., if one omits Assumption \ref{as:P_x0}? Since $D(x)$ is diagonalizable, there exists an $x$-dependent diagonalizing matrix $P(x)$ satisfying $$P(x)D(x)P^{-1}(x)= {\rm diag}(c_1 (x), c_2(x)) \equiv \Lambda (x) \, . $$ Defining $(u,v)^{\top}\equiv P(x)\vec z$, then \eqref{eq:2sys_gen} can be rewritten as
\begin{subequations}\label{eq:2sys_wsource}
\begin{equation}
\begin{cases}{}
u_t (t,x;\omega) + c_1 (x) u_x = g_1(x, u,v) \, ,
\\
v_t (t,x;\omega) +c_2 (x) v_x = g_2 (x, u,v) \, ,\\
 u(0,x,\omega) =  u_{(0,x)}(\omega)  \, , \\
 v(0,x,\omega) =  v_{(0,x)}(\omega)  \, ,
\end{cases} \\
\end{equation}
where the linear source terms are 
\begin{equation}\begin{pmatrix}
    g_1 (x, u,v) \\ g_2 (x, u,v)
\end{pmatrix} \equiv \Lambda(x)P'(x)P^{-1}(x) \begin{pmatrix}
    u\\ v
\end{pmatrix} \, .   
\end{equation}
\end{subequations}
While we do not know how to treat the general case, there is a relaxation of Assumption \ref{as:P_x0} we can consider:
\begin{proposition}\label{prop:2sys_wsource}
    Consider \eqref{eq:2sys_wsource} and assume that $\Lambda(x)P'(x)P^{-1}(x)$ is a diagonal matrix. Then the PDF of the two-points statistics (see \eqref{eq:2ptPhi}) satisfies the following PDE:
$$\partial_t f(t,x,y,U,V) + c_1 (x) \partial _x f +c_2 (y)\partial_y f - \partial_U \left( f\cdot g_1 (x,U) \right) -\partial_V \left( f\cdot g_2 (y,V) \right)=0\, .$$
\end{proposition}

 \begin{proof}
In expanding $\int_{\Omega} \partial_U \psi \partial _t u \, \rd \mu (\omega)$ and $\int_{\Omega} \partial_V \psi \partial _t v \, \rd \mu (\omega)$, two new terms are added to \eqref{eq:expanding_2pt}:
\begin{align}\label{eq:expand_sources}
&\int\limits_{\Omega} \partial_U \psi \left( u(t,x;\omega), v(t,y;\omega) \right) \, g_1 \left(x, u(t,x;\omega), v(t,x;\omega) \right) \, \rd \mu(\omega)\\
+ &\int\limits_{\Omega} \partial_V \psi \left( u(t,x;\omega), v(t,y;\omega) \right) \, g_2 \left(y, u(t,y;\omega), v(t,y;\omega) \right) \, \rd \mu(\omega)  \, . \nonumber
\end{align}
Using the assumption that $\Lambda(x)P'(x)P^{-1}(x)$ is diagonal, \eqref{eq:expand_sources} simplifies to
\begin{align*}
    &\int\limits_{\Omega} \partial_U \psi \left( u(t,x;\omega), v(t,y;\omega) \right)  g_1 \left(x, u(t,x;\omega) \right) \, \rd \mu(\omega) + \int\limits_{\Omega} \partial_V \psi \left( u(t,x;\omega), v(t,y;\omega) \right)  g_2 \left(y,  v(t,y;\omega) \right) \, \rd \mu(\omega) \\
    &= \iint\limits_{\RR ^2}  \partial_U \psi (U,V)   g_1 \left(x, U \right) f(t,x,y,U,V)\, \rd U \rd V  + \iint\limits_{\RR ^2} \partial_V \psi (U,V) \cdot g_2( y,V)  f(t,x,y,U,V)\, \rd U \rd V  \\
    &= -\iint\limits_{\RR ^2}   \psi (U,V)  \partial_U  \left[ g_1 \left(x, U \right) f(t,x,y,U,V) \right]\, \rd U \rd V  - \iint\limits_{\RR ^2}  \psi (U,V) \partial_V \left[ g_2( y,V)  f(t,x,y,U,V)\right]\, \rd U \rd V   \, ,
\end{align*}
where the last passage is due to integration by parts and the vanishing of the boundary terms. Finally, since these passages hold for all test functions $\psi$, we get the PDE above.
 \end{proof}  
\begin{remark}\label{rem:2sysPx}
The proof of Proposition \ref{prop:2sys_wsource} also reveals the difficulty with \eqref{eq:2sys_wsource} when  $\Lambda(x)P'(x)P^{-1}(x)$ is not diagonal: since $g_1$ depends on $V$ (and analogously $g_2$ depends on $U$), then \eqref{eq:expand_sources} depends on $v(t,x;\omega)$ in a non-trivial way. Given the definition of the two points statistics $\phi_{t,x,y}$ (and its PDF $f$), it does not convey information about the joint probability of $u(t,x;\omega)$, $v(t,y;\omega)$ and $v(t,x;\omega)$, and therefore we cannot proceed further.
\end{remark}

Our proofs above generalize without any change to the case of $N$-systems of linear hyperbolic equations.
\begin{proposition}\label{prop:Npoints}
    Consider a system of $N>1$ linear hyperbolic equations with random initial data
    \begin{equation}\label{eq:Nsys}
\begin{cases}{}
\vec z_t (t,x;\omega) + D(x) \vec z_x = 0 \, ,
\\  \vec z(0,x,\omega) =  \vec z_{(0,x)}(\omega)  \, , 
\end{cases}
\end{equation}
where $D(x)$ can be diagonalized by an $x-$indepndent matrix $P$, or equivalently a system of the form
\begin{equation*}
\begin{cases}{}
 u^i_t (t,x;\omega) + c_i(x)  u^i_x = 0 \, ,
\\   u^i(0,x,\omega) =  u^i_{(0,x)}(\omega)  \, , 
\end{cases} \, , \qquad 1\leq i \leq N \, .
\end{equation*}
Define the $N$-points pointwise statistics as the Young measure satisfying, for every measurable $A\subseteq \RR^N$,
$$\phi_{t,x_1,\ldots, x_N} (A) \equiv \mu \left(\left\{\omega \in \Omega ~~ | ~~ \left( u^1 (t,x_1;\omega), \ldots, u^N(t,x_N;\omega) \right)^{\top}\in A \right\}\right) \, .$$
Then the density $f(t,x_1,\ldots ,x_N, U_1, \ldots, U_N)$ satisfies the PDE 
\begin{equation}\label{eq:highdimfsys}
    \partial_t f(t,x_1,\ldots, x_N, U_1, \ldots, U_N) + \sum\limits_{i=1}^N c_i (x_i)\partial_{x_i}f = 0 \, .
\end{equation}
\end{proposition}

\revO{The appearance of an $(N+1)$-dimensional PDE in Proposition \ref{prop:Npoints} is not an artifact of the method, but rather unavoidable, as demonstrated by the counterexample in Section \ref{sec:failure_systems} From a computational perspective, solving \eqref{eq:highdimfsys} becomes increasingly challenging as $N$ grows, due to the well-known curse of dimensionality. For moderate dimensions, such as $N=2,3,4$, which arise in many applications,  \eqref{eq:highdimfsys}  can be efficiently treated using deterministic approaches, including sparse grid discretizations; see, e.g.,
\cite{Griebel2001,BungartzGriebel2004,GriebelVogt2010,Lubich2015}. For very large values of $N$, however, scalable and highly parallelizable Monte Carlo, particle-based, or conversely surrogate model approaches to \eqref{eq:Nsys}  provide the only practical computational alternatives; see, e.g., \cite{Spohn1991,JinPareschiToscani2000,Golse2016,KochLubich2007}.} 

\section{Regularity of the PDF and density estimation}\label{sec:reg}
Kernel Density Estimators (KDEs) are a class of nonparametric statistical methods for density estimation. As mentioned in the introduction, for random initial value problems such as \eqref{eq:genSet}, using KDEs (or any other Monte Carlo-based method) is a standard technique, especially when the dimension of the parameter space, $\Omega$, is large. Here, we ask
\begin{question}\label{q:kde}
    Given $N$ samples from $\mu$, and solving \eqref{eq:genSet} to obtain $u(t,x;\omega)$ for each such sample, what is the expected accuracy of a KDE in estimating $f(t,x;U)$ or $F(t,x;U)$?
\end{question}

Let us first briefly recall what KDEs are (see \cite{izenman1991review, tsybakov2009nonparametric, wasserman2006all} for broader expositions). Let $p\in L^1(\RR) $ be a PDF, let $K\in L^1(\RR)$ be a non-negative {\em kernel} function, and choose a {\em bandwidth} parameter $h>0$. Then, given $N\geq 1$ identically and independently drawn (i.i.d.) samples $y_1, \ldots y_N \sim p$, the associated KDE is 
\begin{equation}\label{eq:kde}
\hat{p}(y) = \hat{p}^{N,K,h}(y) \equiv \frac{1}{N} \sum\limits_{j=1}^N K\left(\frac{y-y_j}{h} \right)\, , 
\end{equation}
for all $y\in \RR$. The expected error of $\hat{p}$ is a standard result \cite[Theorem 1.1]{tsybakov2009nonparametric}, presented here in simplified form:
\begin{theorem}[Tsybakov \cite{tsybakov2009nonparametric}]\label{thm:tsybakov}
    Let $p$ be a PDF with $\beta\geq 1$ continuous derivatives and $\|p^{(\beta)}\|_{\infty}<L$, choose a kernel $K\in L^2( \RR)$ with $\int_{\RR} K(u)\, du =1$, and  any $\alpha >0$. Given $N$ i.i.d., samples from the target, choose $h= h(N;\alpha)= \alpha N^{-1/(2\beta +1)}$. Then\begin{equation}\label{eq:tsybakov} \sup\limits_{y\in \RR}\mathbb{E}|p(y)-\hat{p}(y)|^2 \leq C N^{-\frac{2\beta}{2\beta+1}} \, , \qquad C=C(\alpha, \beta ,L, K) >0 \, ,
    \end{equation}
    where the expectations are with respect to the realizations of the samples. Furthermore, for any fixed $\alpha>0$ the error rate in \eqref{eq:tsybakov} as a function of $N$ is optimal.
\end{theorem}

Hence, to apply \eqref{eq:tsybakov} to our settings, we need to know the regularity of $f(t,x,U)$ (or $F(t,x,U)$) a-priori.
  Given a general initial-value problem with random initial conditions, however, it is not trivial to a-priori guarantee the regularity of the PDF or CDF of the solution. Such guarantees can only be obtained for particular equations, see e.g., \cite{cohen2010convergence, jahnke2022multilevel} for results on the regularity of the solution with respect to the parameters. We show next that for hyperbolic conservation laws, Theorems \ref{thm:cxPDF} and \ref{thm:CDF_PDE} provide the necessary regularity guarantees, and therefore a quantitative answer to Question \ref{q:kde} which only depends on $\phi_{0,x}$.

Consider first the linear scalar settings of \eqref{eq:linScalar}. The following regularity statement is a direct corollary of Theorem \ref{thm:cxPDF}.
\begin{lemma}\label{lem:regLin}
    Under the settings of  Theorem \ref{thm:cxPDF}, if $f_0(x,\cdot) \in C^{\beta}(\RR_U)$ for some $\beta\in \mathbb{N}$ and for all $x\in \RR$, then $f(t,y,\cdot)\in C^{\beta} (\RR_U)$ for all $t,y\in \RR$.
\end{lemma}
    
\begin{proof}
    Note that solutions to \eqref{eq:PDF_utcxux} are traveling waves in $t$ and $x$. While the speed $c(x)$ might not be constant, it does not depend on $U$. Hence, for every $t,x \in \RR$ there exists $x_0=x_0 (t,x)\in \RR$ such that $f(t,x ,U) = f_0 (x_0,U)$. Note further that $x_0= x_0(t,x)$ is determined by the method of characteristics, and therefore (by regularity theory of ODEs) it is a continuous function of $x$ for every fixed $t \geq  0$. Therefore, while  $\|f(t,x,\cdot)\|_{C^r(\RR_U)}$ can generally depend on $t$, it cannot blow-up in finite time.
\end{proof}

With this regularity result and the statistical theory of \eqref{eq:tsybakov}, we can  establish the statistical approximation error of applying KDEs to compute $f(t,x,U)$ in the linear scalar settings of \eqref{eq:linScalar}; 
\begin{proposition}\label{prop:KDE_rate}
    Suppose $f_0(\cdot,U) \in C^1(\RR _x)$ for all $U\in \RR$, and that there exists $\beta \in \mathbb{N}$ and $L>0$ such that $$\sup\limits_{x\in \RR} \sup\limits_{U\in \RR}\left| \partial_U^{(\beta)} f_0(x;U)\right| < L \, .$$  Let $\hat{f} (t,x,\cdot)$ be the KDE of $f(t,x;\cdot)$ using $N\geq 1$ i.i.d.\ samples from $\mu$ and solving \eqref{eq:linScalar} for each sample, see \eqref{eq:kde}. Then, for every $x,U\in \RR$ and every $T\geq 0$, there exists $C(T)\geq 0$ such that     \begin{equation*}
    \sup\limits_{x\in \RR, t\in [0,T]} \sup\limits_{U\in \RR_U} \mathbb{E}|f(t,x;U) - \hat{f} (t,x;U ) |^2 \leq C(T) N^{-\frac{2\beta}{2\beta+1}} \, .
    \end{equation*}
\end{proposition}
\begin{proof}
    For any fixed $t\geq 0$ and $x\in \RR$, an upper bound on $\sup_{U\in \RR_U} \mathbb{E}|f(t,x;U) - \hat{f} (t,x;U ) |^2$  is a straightforward application of \eqref{eq:tsybakov} to $p=f(t,x;\cdot)$, combined with the regularity result in Lemma \ref{lem:regLin}. For $t=0$, a uniform bound in  $x$ is a consequence of the fact that $\partial_U^{(\beta)} f_0(x,U)$ is uniformly bounded in $x$ as well. The uniform bound over all $t\in [0,T]$ is a consequence of the fact that the solution $f(t,x;U)$ maintains its regularity for all $t \geq 0$, and satisfies an upper bound of the form $$\sup\limits_{x\in \RR} \sup\limits_{U\in \RR}\left| \partial_U^{(\beta)} f(t,x;U)\right| < L(t) \, ,$$
     for some (potentially increasing, but always finite) $L:[0,\infty) \to [0,\infty)$. This concludes the proof. Whether or not the interval $[0,T]$ can be taken to be $[0,\infty)$ depends on whether $L(t)\to \infty$ as $t\to \infty$, which in turn depends on the particulars of the medium heterogeneity function $c(x)$ in \eqref{eq:linScalar}.
\end{proof}

Now, we turn to the nonlinear settings of \eqref{eqn:cl}. Here, it is preferable to discuss CDF estimation (i.e., the statistical approximation of $F(t,x,U)$ based on samples of $u(t,x;\omega)$), since the regularity theory of $F$ is more straightforward.
\begin{lemma}
    Assume the conditions of Theorem \ref{thm:CDF_PDE}, and suppose that $F_0$ is $\beta\in \mathbb{N}$ differentiable in both $x$ and $U$. Then there exists $C>0$, independent of $t$, such that  $$\sup\limits_{x\in \RR} \sup\limits_{U\in \RR} \|\partial^{\beta}_U F(t,x,U) \|_{\infty} \leq C t^\beta \, ,$$
    for all $t\geq 0$ for which Theorem \ref{thm:CDF_PDE} is valid.
\end{lemma}
\begin{proof}
    Since $F(t,x,U)=F(0,x-a'(U)t,U)$, and denoting the partial derivatives of $F$ with respect to the second and third coordinates by $F_x$ and $F_U$, respectively, we have
    \begin{align*}
        \partial_U F(t,x,U) &= \partial _U F(0,x-a'(U)t,U) \\
        &= -F_x (0,x-a'(U)t,U)  a''(U) t + F_u (0,x-a'(U)t,U) \, .
    \end{align*}
    Hence, if $F(0,\cdot, \cdot)$ is differentiable in $x$ and $U$, then so is $F(t,\cdot ,\cdot)$. Similarly by induction, after $\beta$ derivatives, $\partial_U^{\beta} F(t,x,U)$ is a sum of powers of $t$, the largest one is of the form $$\partial_x ^{\beta}F(0,x-a'(U)t,U) \left( a''(U)  t\right)^\beta \, ,$$ which completes the upper bound for large times $t\gg 1$.
\end{proof}
The Kernel estimator of the CDF $\hat{F}(y)$ is the cumulative integral of the KDE $\hat{p}$, see \eqref{eq:kde}, on the interval $(\infty, y)$. The theory of these estimators is a consequence of the theory of KDEs, and we refer to \cite{berg2009cdf, wang1991nonparametric} and the references therein for details. In much the same way as Proposition \ref{prop:KDE_rate}, we obtain:
\begin{proposition}
    Let $\hat{F} (t,x,\cdot)$ be the Kernel estimator of $F(t,x;\cdot)$ using $N\geq 1$ i.i.d.\ samples from $\mu$ and solving \eqref{eqn:cl} for each sample. Then, for every $x,U\in \RR$ and every $T\geq 0$, there exists $C>0$ such that     \begin{equation*}
    \sup\limits_{x\in \RR, t\in [0,T]} \sup\limits_{U\in \RR_U} \mathbb{E}|F(t,x;U) - \hat{F}^N (t,x;U ) |^2 \leq C T^{\beta} N^{-\frac{\beta+1}{2\beta+1}} \, ,
    \end{equation*}
    under a choice of optimal Kernel width as specified in \cite{berg2009cdf, wang1991nonparametric}.
\end{proposition}

\section{Numerical examples}

\subsection{Linear scalar case}
Let us first demonstrate that, for a linear hyperbolic equation, the evolution of the pointwise statistics, $\phi_{t,x}$, through the correponding PDFs $f(t,x,U)$, follow an evolution law which is {\it (i)} independent of the particular underlying random field $u(t,x;\omega)$, and {\it (ii)} described by Theorem \ref{thm:cxPDF}. 

\subsubsection{Independence of spatial correlations}\label{sec:test1}
Here, we consider two different random initial value problems, satisfying the same PDE. In both cases, $\phi_{0,x}$ are the same, but $u(0,x;\omega)$ is different from $v(0,x;\omega)$. 
\begin{equation} \label{eqn:u0}
\begin{cases}
\partial_t u(t,x; \omega ) + c(x)\partial_x u = 0 \, , \quad c(x)=  0.1x^2+1, \quad x \in [-3,4] \, ,
\\ u(0,x,\omega) = \omega (g(x+1) + g(x-1) )\, , \quad  g(x) = \max\{1-(3x)^2,0\}, \quad \omega \sim \text{unif}~[0,1] \, ,
\end{cases}
\end{equation}

\begin{equation} \label{eqn:v0}
\begin{cases}
\partial_t v(t,x; \omega ) + c(x)\partial_x v  = 0\, , \quad c(x)=  0.1x^2+1, \quad x \in [-3,4] \, ,
\\ v(0,x,\omega) = \omega g(x+1) + (1-\omega)g(x-1) \, , \quad  g(x) = \max\{1-(3x)^2,0\}, \quad \omega \sim \text{unif}~[0,1] \, .
\end{cases}
\end{equation}
Not only do \eqref{eqn:u0} and \eqref{eqn:v0} {\em look} different, they represent different random fields. For example: the probability that $v(0,x;\omega)=v(0,-x;\omega)$ is zero for $x\in (-4/3,-2/3)\cup (2/3,4/3)$, whereas $u(0,x;\omega)=u(0,-x;\omega)$ almost surely for these same values of $x$.

Nevertheless, denoting the pointwise statistics corresponding to $u$ and $v$ by $\phi_{t,x}^u$ and $\phi_{t,x}^v$, one can verify that $\phi^u _{0,x}=\phi^{v}_{0,x}$ for every $x\in \RR$. Indeed, denoting the respective cumulative distribution function  by $F^u$ and $F^v$, we therefore have a single function $F_0 :\RR_x \to L^1(\RR_U)$ with $F_0=F^u(0,\cdot ,\cdot)=F^v(0,\cdot ,\cdot)$. 
Hence, by Theorem~\ref{thm:CDF_PDE}, we expect that $F^v$ and $F^u$ both satisfy the following PDE 
\begin{equation} \label{eqn:f000}
\begin{cases}
\partial_t F(t,x;U ) + c(x)\partial_x F(t,x;U) =0\, ,
\\ F(0,x,U) = F_0(x,U)  \, .
\end{cases}
\end{equation}
The initial condition for CDF is:
\begin{itemize}
    \item for $x \in [-\frac{2}{3}, \frac{2}{3}]$ or $x\leq -\frac{4}{3}$ or $x \geq \frac{4}{3}$:
    \begin{align*}
    F(0,x,U) = \left\{ \begin{array}{cc} 
     0, & U<0, \\ 1, & U>0;
    \end{array}\right.
    \end{align*}
   \item for $x\in  [-\frac{4}{3}, -\frac{2}{3}]$:
   \begin{align*}
    F(0,x,U) = \left\{ \begin{array}{cc} 
     0, & U<0, \\ \frac{U}{g(x+1)}, & 0\leq U \leq g(x+1),
     \\ 1, & U \geq g(x+1);
    \end{array}\right.
    \end{align*}
    \item for $x \in  [\frac{2}{3}, \frac{4}{3}] $:
    \begin{align*}
    F(0,x,U) = \left\{ \begin{array}{cc} 
     0, & U<0, \\ \frac{U}{g(x-1)}, & 0\leq U \leq g(x-1), \\ 1, & U \geq g(x-1).
    \end{array}\right.
    \end{align*}
\end{itemize}
 \revO{The initial data for PDF in \eqref{eq:PDF_utcxux} can be derived accordingly by taking the derivative of $F(0,x,U)$ with respect to U.}

In practice, we solve \eqref{eqn:u0} and \eqref{eqn:v0} using the Monte Carlo method in $w$ with $10^5$ samples and reconstruct the PDF and CDF using a KDE for every x.\footnote{We use Matlab's built-in function \texttt{ksdensity}.} For \eqref{eqn:u0} and \eqref{eqn:v0}, the spatial and temporal derivatives are treated using the upwind method, with numerical parameters $\Delta x = 0.01$, $\Delta t = 10^{-3}$. For \eqref{eqn:f000}, we also use an upwind scheme with a slope limiter, applying the same discretization parameters. We discretize in $U$ with $\Delta U = 0.0151$.

The results are collected in Figs.~\ref{fig:linear1} and \ref{fig:linear1-pdf}. Fig.~\ref{fig:linear1} shows the CDF obtained from the three models individually, and the last plot displays the error between $\eqref{eqn:u0}$ and $\eqref{eqn:v0}$. The $L^1$ error between \eqref{eqn:u0} and \eqref{eqn:f000} is $0.0438$. Fig.~\ref{fig:linear1-pdf} further compares the mean and variance across the three models--\revO{\eqref{eqn:u0}, \eqref{eqn:v0} and \eqref{eq:PDF_utcxux}}, with errors recorded in the caption. In all, we observe an excellent match between the Monte-Carlo computed pointwise statistics of $u$ and $v$, and with the directly-evolved \revO{$f(t,x;U)$}, as predicted.
\begin{figure}[h!]
    \centering
    \includegraphics[scale= 0.5]{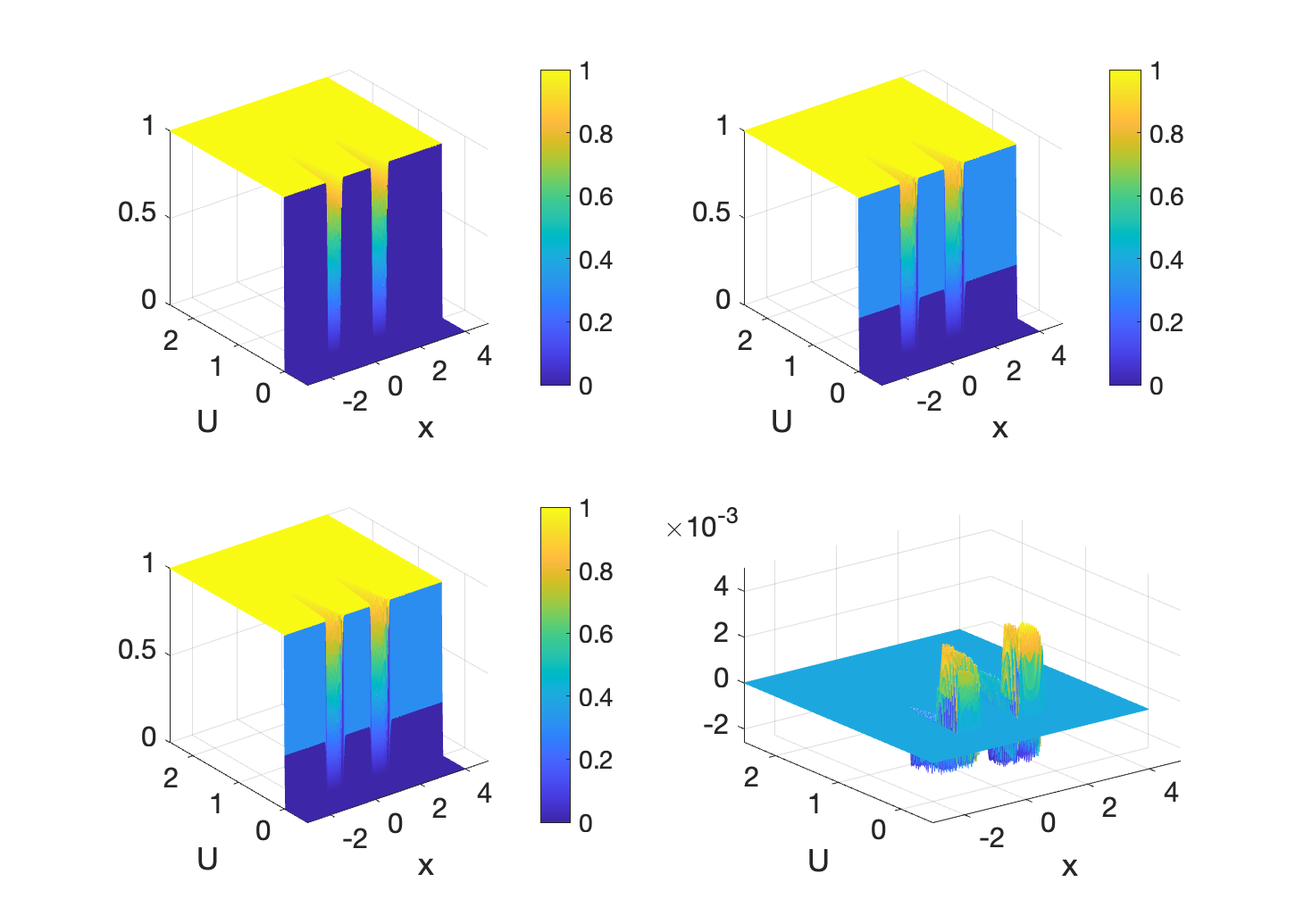}
    \caption{Scalar, linear hyperbolic equations:  CDF $F(0.2,x,U)$ corresponding to, from left to right, top to bottom \eqref{eqn:f000}, \eqref{eqn:u0} and \eqref{eqn:v0}. The last plot is the error between \eqref{eqn:u0} and \eqref{eqn:v0}. The $L^1$ error between the solution to \eqref{eqn:f000} and \eqref{eqn:u0} is 0.0438.}
    \label{fig:linear1}
\end{figure}
\begin{figure}[h!]
    \centering
    \includegraphics[width= 0.8\textwidth, height= 0.25\textwidth]{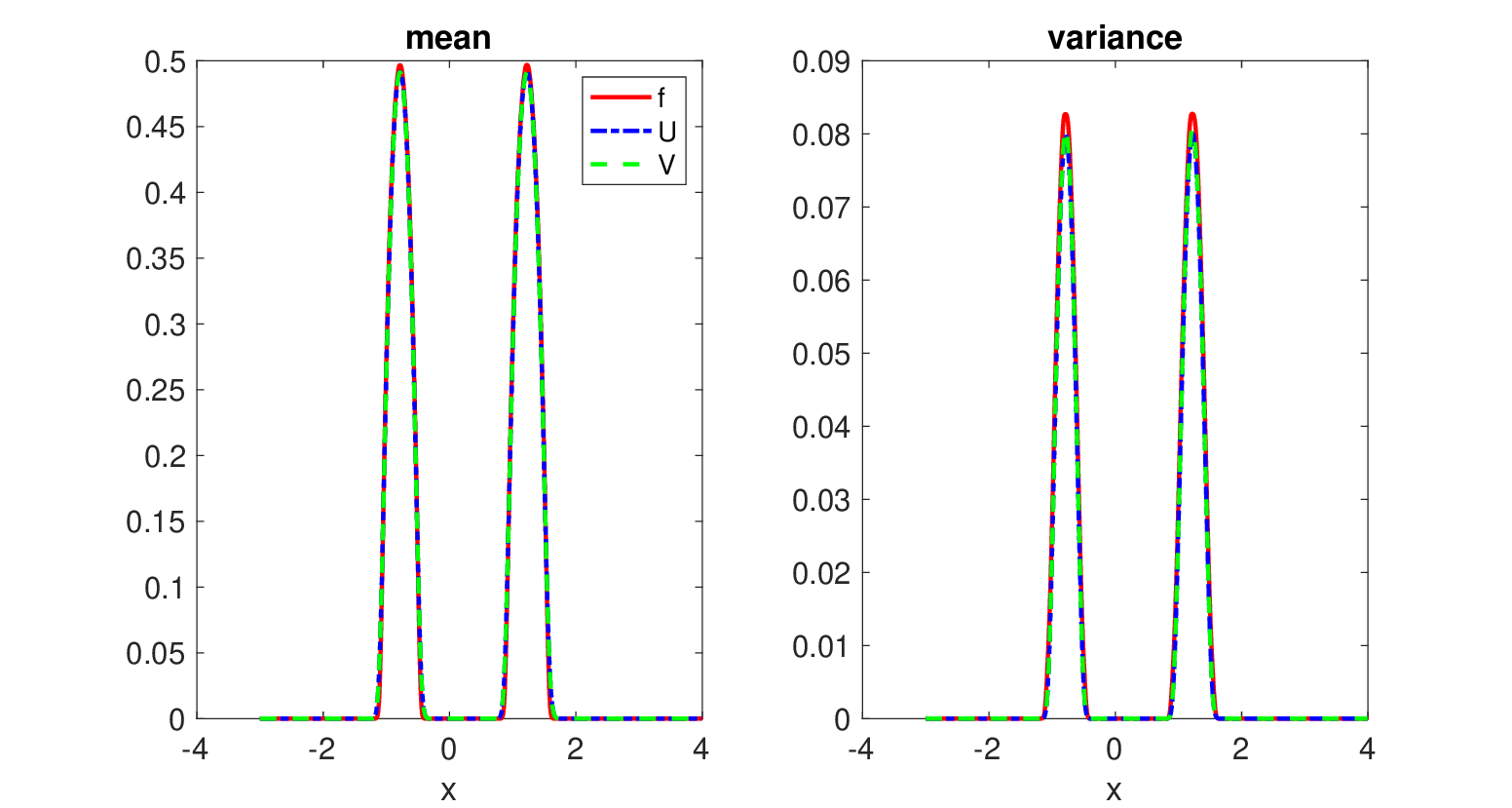}
    \caption{Same settings as in Fig.\ \ref{fig:linear1}. Comparison of mean and variance. The $L^1$ error in mean between $U$ and $V$ is $5.7607e-04$; and between $U$ and $f$ is $0.0136$. Likewise, the  $L^1$ error in mean between $U$ and $V$ is $4.8364e-04$; and between $V$ and $f$ is $0.0044$.}
    \label{fig:linear1-pdf}
\end{figure}

\subsubsection{Two-dimensional example} \label{sec:linear2}
As noted in the introduction, having direct evolution equations for the pointwise statistics, $\phi_{t,x}$, implies that even if the underlying parameter space $\Omega$ is high-dimensional, we can solve the (fixed-dimensional) equations for the PDF or CDF instead. To see this numerically, consider
\begin{equation} \label{eqn:u}
\begin{cases}
\partial_t u(t,x; \vec\gamma )= c(x)\partial_x u \, , ~ c(x) = 1- 0.1x^2,\\
u(0,x,\vec \gamma ) = u_{(0,x)}(\vec \gamma) := \gamma_1 g_1(x) +  \gamma_2 g_2(x)\,,
\end{cases}
\end{equation}
where $\vec\gamma = (\gamma_1, \gamma_2)\in \RR^2$ is distributed according to the normal distribution with mean $(0,1)^{\top}$ and covariance matrix ${\rm diag}(1, 0.25)$,
and where  $g_1(x) = e^{10(x+1)^2} + 1$ and 
 $g_2(x) = e^{10(x-1)^2} + 1$. 
By Theorem~\ref{thm:cxPDF}, $f$ should satisfy \eqref{eq:PDF_utcxux}
with
$    f_{(0,x)}(U) = \mathcal{N} (g_2(x), g_1(x)^2+ 0.5^2 g_2(x)^2 )$
for any $x$ in the computational domain.

Fig.~\ref{fig:linear2_1} compares the PDFs obtained by a Monte Carlo simulation of \eqref{eqn:u} and by solving the PDF evolution equation \eqref{eq:PDF_utcxux}.  Similarly, Fig.\ \ref{fig:linear2_2} compares the mean and variance obtained from these two models, showing good agreement. The $L^1$ errors between them are 0.0266 and 0.0554, respectively. The right plot in Fig.\ \ref{fig:linear2_2} displays $f(0.4,0.29,U)$, where we observe a good agreement between the results. \revO{To explain these errors, first note that Note that in our method, since $f(0,x,U)$ can be computed in closed form, the only error is that associated with solving the PDE for $f$, \eqref{eq:PDF_utcxux}. In this one-dimensional PDE, this error can be controlled relatively well with the numerical parameters we report below. Hence, the sources of the error is rather due to the Monte Carlo KDE estimator. The error of such nonparametric statistical estimators converges quite slowly in the number of samples (see Theorem \ref{thm:tsybakov}). Indeed, we also observe numerically, that increasing the number of samples dramatically does reduce the overall error (results not shown).} \st{Here, better precision can be obtained by drawing more samples from the initial conditions, the number of which can always be increased at a relatively low computational cost.}

In our examples, we set $x\in [-2,3]$ and apply a periodic boundary condition. For \eqref{eq:PDF_utcxux},  we use $\Delta x = 0.01$, $\Delta t = 0.002$ and $\Delta U = 0.012$. For the Monte Carlo method applied to \eqref{eqn:u}, we use  $\Delta x = 0.01$, $\Delta t =0.0025$ and $10^5$ samples. The density for \eqref{eqn:u} is reconstructed using \texttt{ksdensity} with kernel width $0.02$.

\begin{figure}[h!]
    \centering
    \includegraphics[width= 1\textwidth, height=0.3 \textwidth]{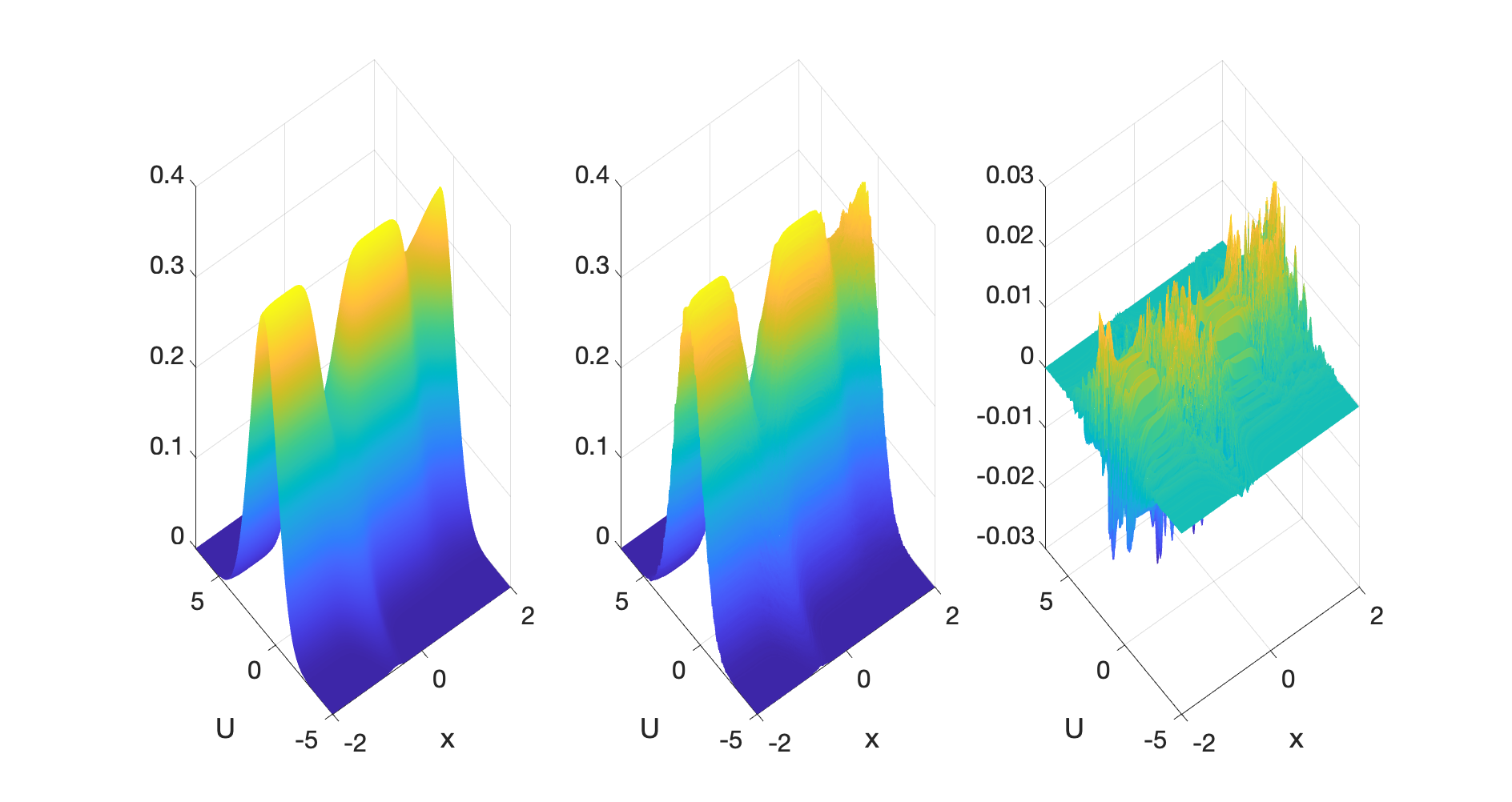}
    \caption{The  PDF $f(t=0.4,x,U)$ for \eqref{eqn:u}. Left is computed using \eqref{eq:PDF_utcxux} and middle is using \eqref{eqn:u} with $10^5$ samples. The right plot is the difference between the two.}
    \label{fig:linear2_1}
\end{figure}

\begin{figure}[h!]
    \centering
    \includegraphics[width= 1.0\textwidth, height= 0.3\textwidth]{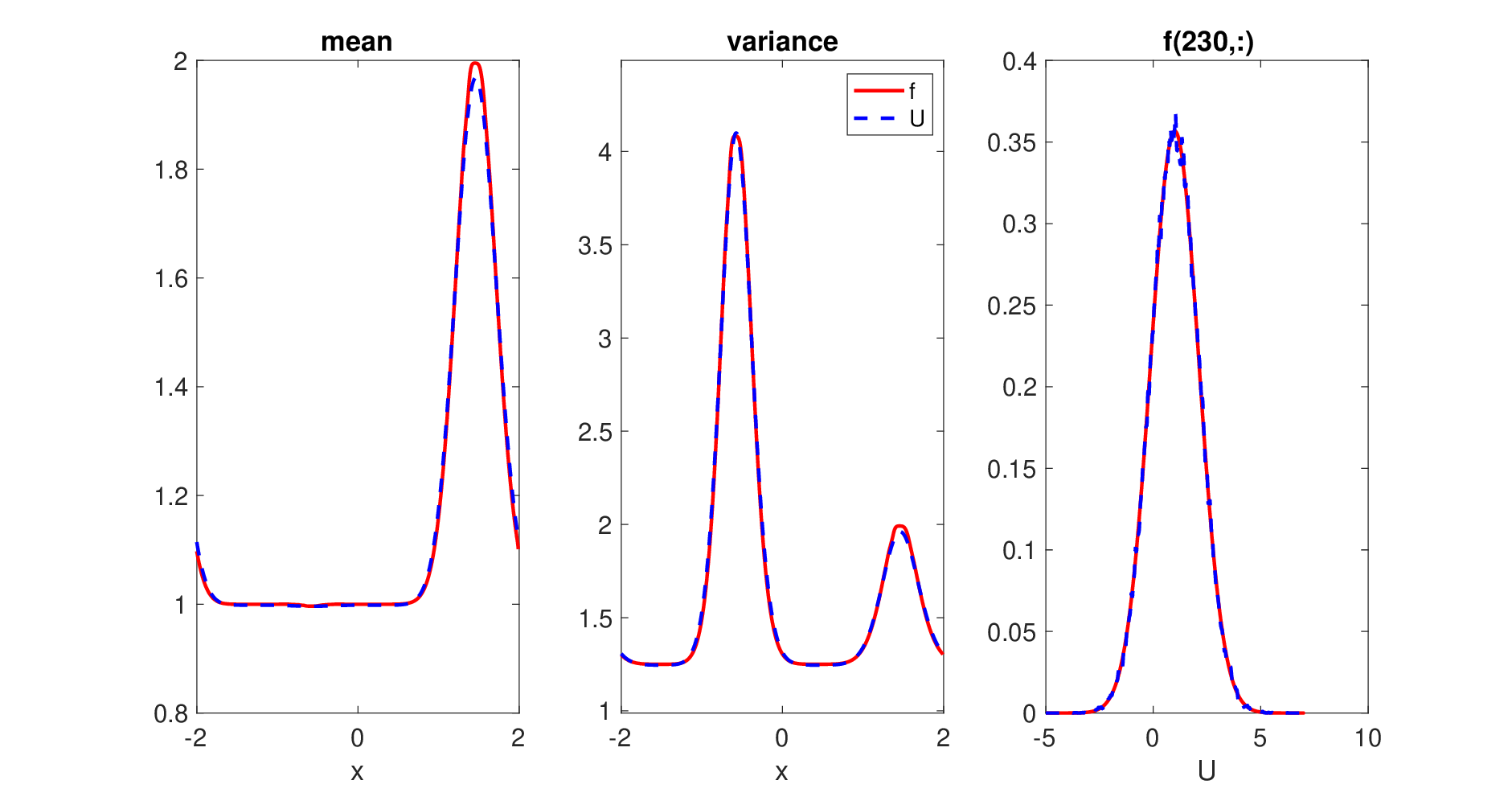}
    \caption{Same settings as in Fig.\ \ref{fig:linear2_1}. Comparison of mean, variance, and PDF (at $x=0.29$) at $t=0.4$, \revO{ using a direct solution of \eqref{eq:PDF_utcxux} (red, solid) and Monte-Carlo simulations of \eqref{eqn:u} (blue, dashed).} The $L^1$ errors between the mean and variance are 0.0266 and 0.0554, respectively.}
    \label{fig:linear2_2}
\end{figure}

\subsection{Nonlinear rarefaction waves}
To demonstrate the validity of Theorems \ref{thm:CDF_PDE} and \ref{thm:PDF_PDE_nl}, we consider the Burgers' equation in the {\em absence of shock}, i.e., the initial data are set such that a rarefaction wave emerges almost surely:
\begin{equation} \label{eqn:u-rare}
\begin{cases}
\partial_t u + \half \partial_x (u^2) = 0 \, ,
\\ u(0,x,\omega) = \omega (g_1(x-1) + g_2(x+1) )\,, \quad \omega \sim \text{unif}~[0,1]\,,
\end{cases}
\end{equation}
where 
\begin{equation} \label{g1g2}
    g_1 (x) = \half (1+\text{tanh}(40x)), \quad g_2(x) = \half(\text{tanh}(40x)-1) = g_1(x)-1\,.
\end{equation}
Here, as in the linear case of Sec.\ \ref{sec:test1}, we will show that only $\phi_{0,x}$ matters to determine $\phi_{t,x}$ for $t>0$. Hence, we consider a different set of random data with the same initial pointwise statistics:
\begin{equation} \label{eqn:v-rare}
\begin{cases}
\partial_t v + \half \partial_x (v^2) = 0 \, ,
\\ u(0,x,\omega) = \omega g_2(x+1) + (1-\omega)g_1(x-1) \,, \quad \omega \sim \text{unif}~[0,1]\,.
\end{cases}
\end{equation}
One can check directly that the CDFs for $u(0,x,\omega)$ and $v(0,x,\omega)$ are identical for each $x$, up to exponentially small terms. To calculate this common CDF, $F_0 (x;U)$, define $[a,b]=[-0.25,0.25]$ as the interval on which $g_1 (x-1)+g_2 (x+1)\approx 0$, where we approximate $g_1 (x-1)\approx 0$ for $x<a$ and $g_2(x+1)\approx 0$ for $x>b$. Therefore, by direct computation,
\begin{itemize}
    \item for $x \in [a, b]$:
    \begin{align*}
    F_0(x,U) = \left\{ \begin{array}{cc} 
     0, & U<0, \\ 1, & U>0;
    \end{array}\right.
    \end{align*}
   \item for $x < a$:
   \begin{align*}
    F_0(x,U) = \left\{ \begin{array}{cc} 
     1, & U \geq 0, \\ 1- \frac{U}{g_2(x+1)}, & 0 \geq U \geq g_2(x+1),\\
     0, & U< g_2(x+1); 
    \end{array}\right.
    \end{align*}
    \item for $x >b $:
    \begin{align*}
    F_0(x,U) = \left\{ \begin{array}{cc} 
     0, & U<0, \\ \frac{U}{g_1(x-1)}, & 0\leq U \leq g_1(x-1),\\
     1, & U \geq g_1(x-1).
    \end{array}\right.
    \end{align*}
\end{itemize}
An alternative way of computing $F_0(x,U)$ is
\begin{align} \label{F0_rare}
    F_0(x,U) = \left\{\begin{array}{cc}  
    \min \{ \max\{ \frac{U}{g_1(x-1) + g_2(x+1)}, 0\}, 1\},  & g_1(x-1) + g_2(x+1)>0,\\
    \max\{ 1- \max \{ \frac{U}{g_1(x-1) + g_2(x+1)} ,0\}, 0\}, &  g_1(x-1) + g_2(x+1) <0,
    \\ \mathbf{1}_{U\geq 0}, & \quad  g_1(x-1) + g_2(x+1) =0,
    \end{array} \right. 
\end{align}
where for any set $B$ denote by $\mathbf{1}_{B}$ the characteristic function of  $B$.

From the expression of $F_0(x,U)$ in \eqref{F0_rare}, we see that $F_0$ is close to a function which is not differentiable everywhere, and therefore $f(t,x;U)=\partial_U F(t,x;U)$ is very large at these point. Therefore, it is more prudent to evolve the CDF equation according to \eqref{eqn:CDF_noshock}, and then compute the PDF from it using central-difference numerical differentiation.  

Both the CDF and the resulting mean and covariance are compared to direct Monte Carlo simulations of \eqref{eqn:u-rare} and \eqref{eqn:v-rare}, as shown in Figs.~\ref{fig:rare-pdf} and \ref{fig:rare}. The Monte Carlo simulations use $10^5$ samples. 
As in the linear case discussed previously, the results are nearly indistinguishable. The discretization parameters used are $\Delta x = 0.01$, $\Delta t = 0.0025$, and $\Delta U = 0.01$.

\begin{figure}[h!]
    \centering
    \includegraphics[width= 0.9\textwidth, height=0.4\textwidth]{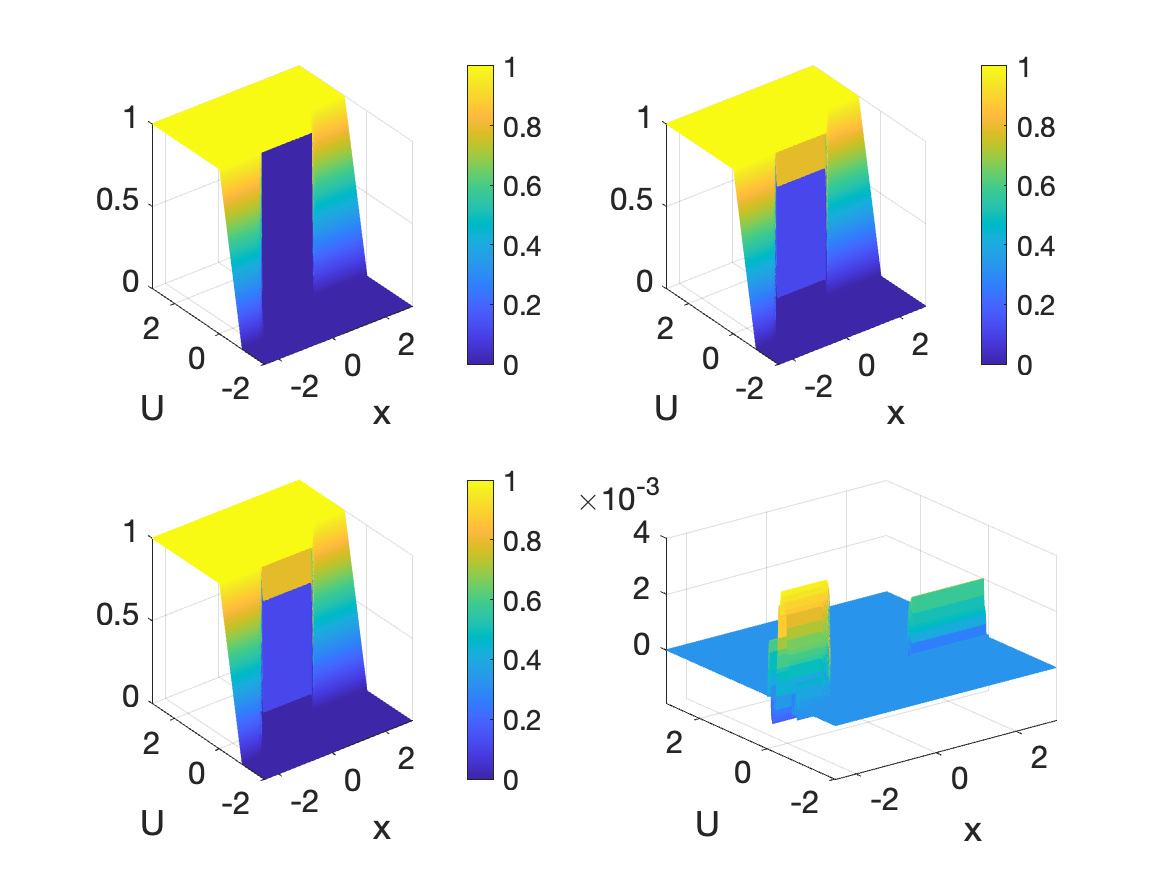}
    \caption{Comparison of CDF of random rarefaction solutions at time $t=0.8$. From top to bottom, left to right, the results are obtained by solving  \eqref{eqn:CDF_noshock} (for $F$), \eqref{eqn:u-rare}, and \eqref{eqn:v-rare}, respectively. The bottom right plot is the difference between \eqref{eqn:u-rare} and \eqref{eqn:v-rare}, and the $L^1$ error between \eqref{eqn:CDF_noshock} and \eqref{eqn:v-rare} is 0.0179.}
    \label{fig:rare-pdf}
\end{figure}

\begin{figure}[h!]
    \centering
    \includegraphics[width= 0.8\textwidth, height=0.3\textwidth]{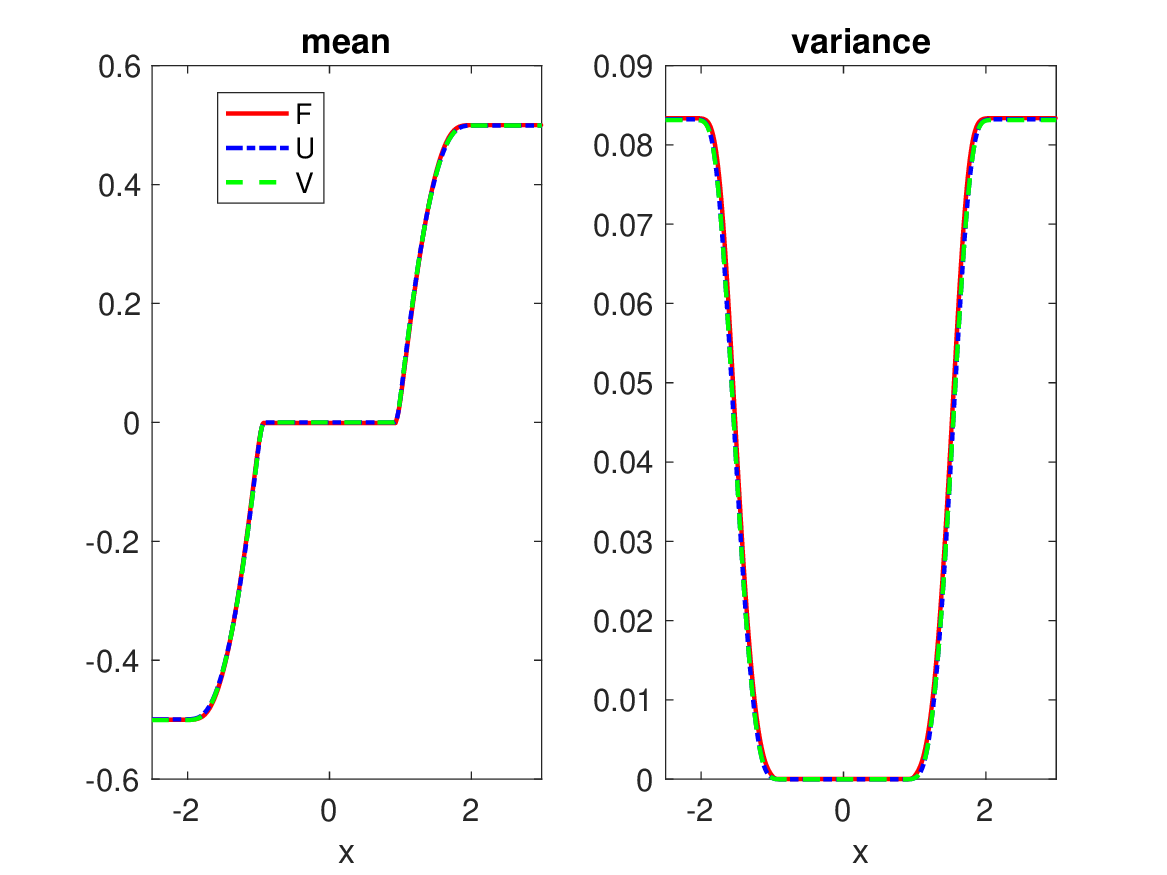}
    \caption{Same settings as Fig.\ \ref{fig:rare-pdf}. Comparison of mean and variance at time $t=0.8$ between three models: u \eqref{eqn:u-rare} (dashed blue), $v$ \eqref{eqn:v-rare} (dashed green), and the directly-evolved CDF, $F$, according to \eqref{eqn:CDF_noshock} (solid red). }
    \label{fig:rare}
\end{figure}

We now turn to demonstrate the non-local PDE \eqref{eqn:PDF_PDE_nl} for evolving the PDF $f$. First, as noted, differentiate the CDF obtained above with respect to $U$ to obtain (a numerical approximation of) $f(t,x;U)$. Alternatively, we compute $f_0 (x,U)$ by differentiating $F_0 (x;U)$ (see  \eqref{F0_rare}) and then evolve the numerically computed PDF according to \eqref{eqn:PDF_PDE_nl}, and compare the two approximations of $f$ at $t=0.8$. The scheme by which we solved \eqref{eqn:PDF_PDE_nl} is described in Appendix \ref{sec:appendix}..

The mean and covariance for these two solutions are shown in Fig.~\ref{fig:rare3}. It is observed that the ``differentiate and evolve" approach (based on \eqref{eqn:PDF_PDE_nl}) is slightly more diffusive than the ``evolve then differentiate" (based on $F(t,x;U)$) approach. 
\begin{figure}[h!]
    \centering
    \includegraphics[width= 0.8\textwidth, height=0.3\textwidth]{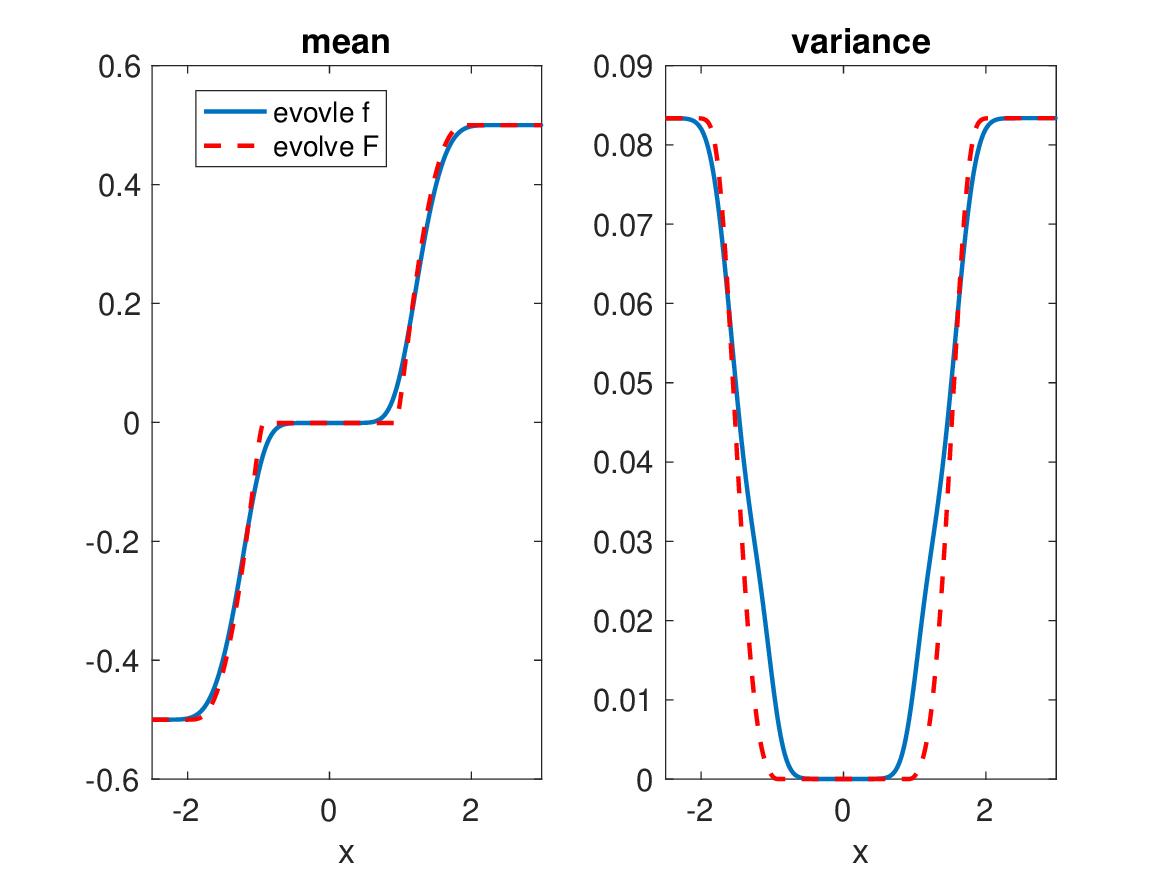}
    \caption{For the rarefaction wave settings of \eqref{eqn:u-rare}, we compare the mean and variance at $t=0.8$ as obtained through evolving the PDF using \eqref{eqn:PDF_PDE_nl} (blue, solid) and by evolving the CDF using \eqref{eqn:CDF_noshock} (red, dashed).}
    \label{fig:rare3}
\end{figure}

\begin{remark}[\revO{Evolving the CDF  vs the PDF - practical considerations.}]
\revO{Computationally, evolving the CDF
is often preferable, for the following reasons: estimating the initial CDF from samples is statistically
more stable than estimating a density. Furthermore, the CDF satisfies a simple linear
transport equation, whereas the evolution equation for the PDF is nonlocal and requires solving in parallel for many values of~$U$. Lastly, the
CDF is well–defined even when the measure $\phi_{t,x}$ has singular components.}

\revO{On the other hand, certain applications require the PDF directly, for instance
when densities are used for visualization or as a basis for sampling algorithms.
A prominent example is Markov Chain Monte Carlo, where the acceptance step
depends on evaluating the target density (up to normalization); see, e.g.,
\cite{brooks_etal_2011, robert_casella_2010}. In such cases, evolving the
PDF avoids numerical differentiation of a CDF, which may introduce additional
error. Both approaches are valid, and the choice depends on the intended use.
}
\end{remark}

\subsection{Nonlinear case: shock}\label{sec:ShockExample}
As in the previous examples, we study the effect of spatial correlations by comparing two distinct initial conditions that share identical pointwise statistics $\phi_{0,x}$. We will demonstrate that, although the pointwise statistics agree before the creation of the shock, they diverge afterwards.

Specifically, the equation for $u(t,x; \omega )$ reads 
\begin{equation} \label{eqn:u-shock}
\begin{cases}
\partial_t u + \half \partial_x (u^2) = 0 \, ,
\\ u(0,x,\omega) = \omega (g_1(x-\half) + g_2(x+\half) )\,, \quad \omega \sim \text{unif}~[0,1]\,,
\end{cases}
\end{equation}
where 
\begin{equation*}
    g_1 (x) = \left\{ \begin{array}{cc}
    0 & x < -\frac{1}{4} \\ -2x+\half & -\frac{1}{4} \leq x < \frac{1}{4} \\ -1 & x \geq \frac{1}{4}\end{array}\right., \qquad g_2(x) = g_1(x)+1\,.
\end{equation*}
The equation for $v(t,x; \omega )$ reads 
\begin{equation} \label{eqn:v-shock}
\begin{cases}
\partial_t v + \half \partial_x (v^2) = 0 \, ,
\\ v(0,x,\omega) = \omega g_2(x+\half) + (1-\omega) g_1(x-\half) \,, \quad \omega \sim \text{unif}~[0,1]\,.
\end{cases}
\end{equation}

\begin{figure}
\centering
\begin{subfigure}[b]{0.32\linewidth}
    \includegraphics[width=\linewidth, height=0.8\linewidth]{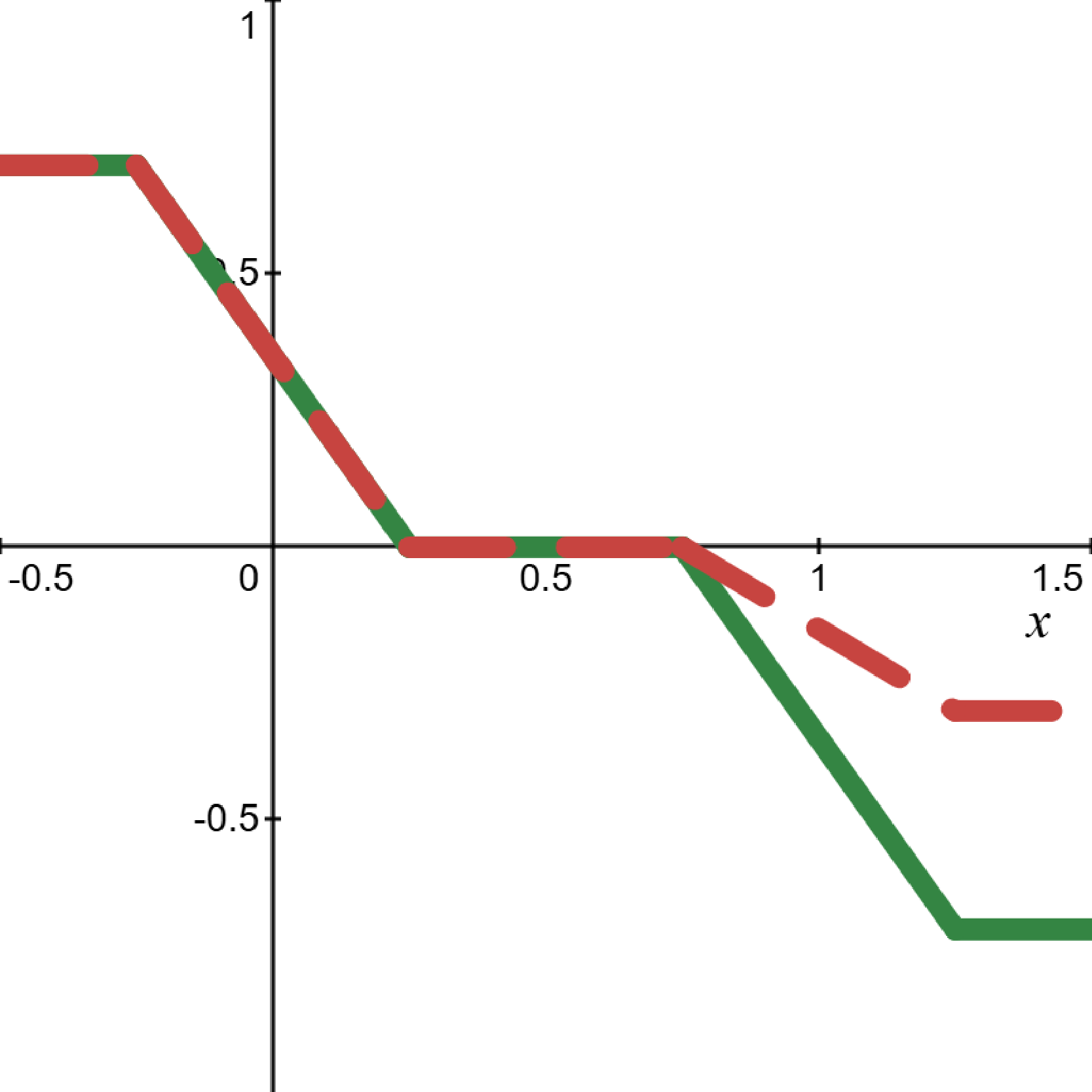}
    \caption{$t=0$}
    \label{fig:u0v0_1}
\end{subfigure}
\hfill
\begin{subfigure}[b]{0.32\linewidth}
    \includegraphics[width=\linewidth, height=0.8\linewidth]{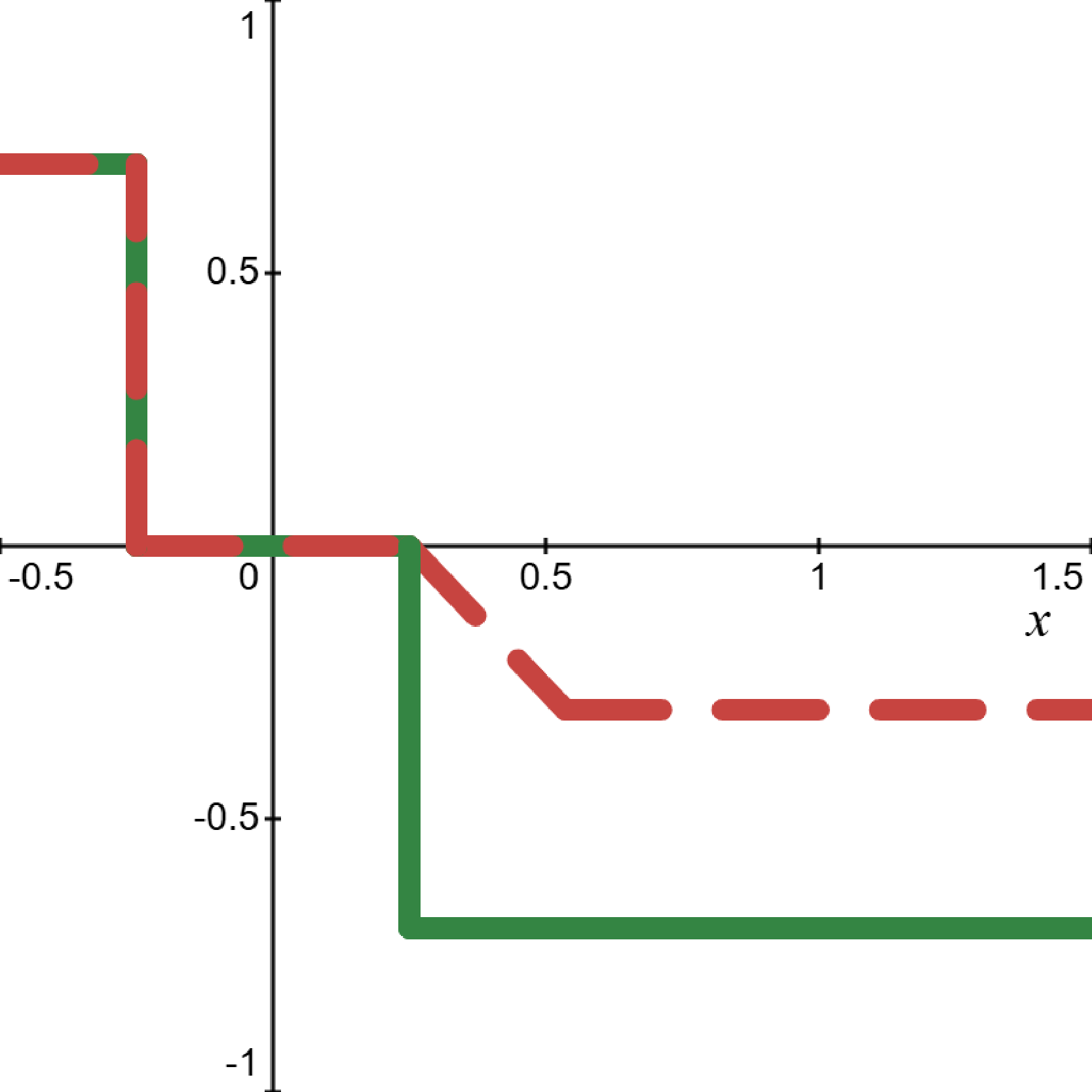}
    \caption{$t_1(0.7)=\frac57$}
    \label{fig:u0v0_2}
\end{subfigure}
\hfill
\begin{subfigure}[b]{0.32\linewidth}
    \includegraphics[width=\linewidth, height=0.8\linewidth]{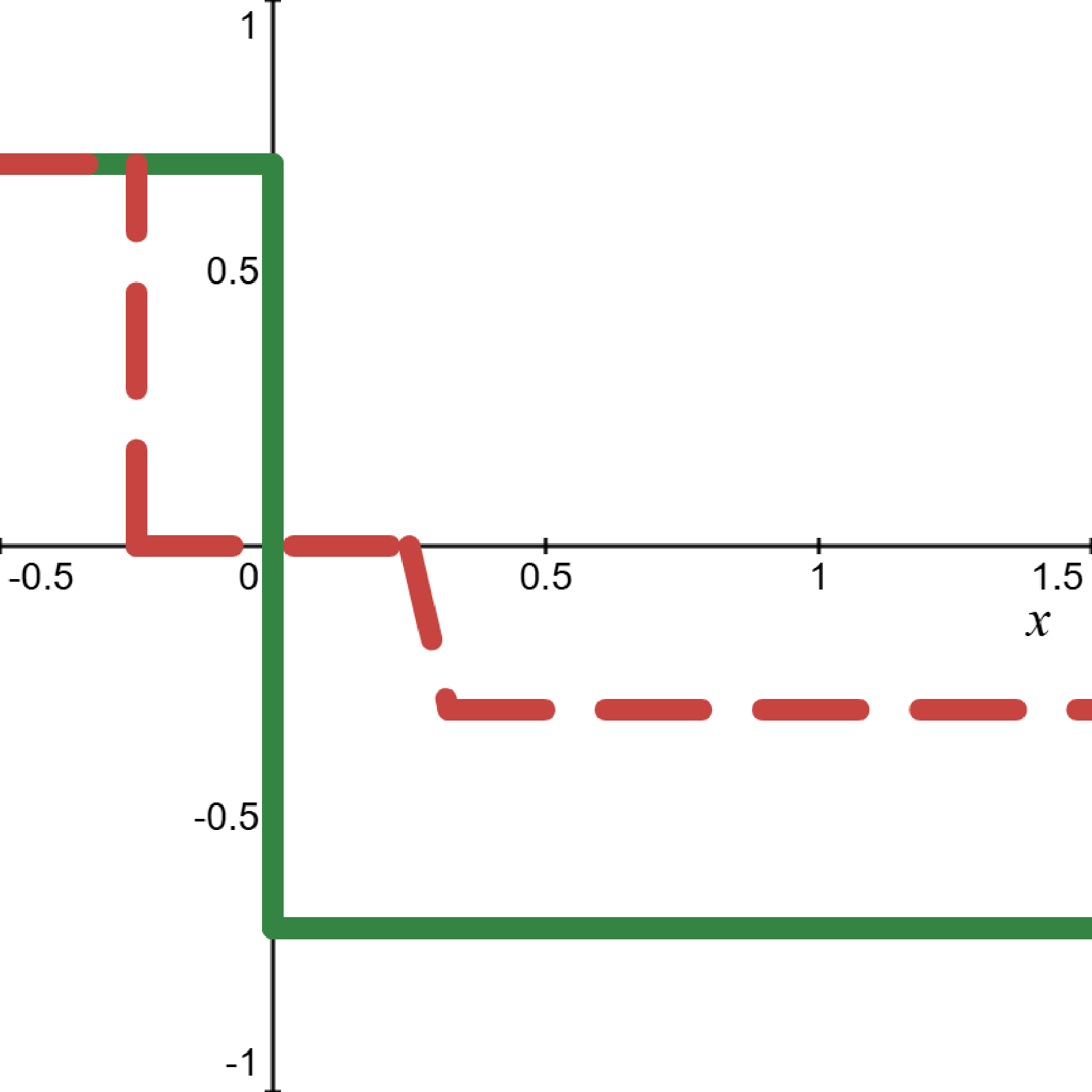}
    \caption{$t_2 (0.7) = \frac{10}{7}$}
    \label{fig:u0v0_3}
\end{subfigure}
   \caption{Illustration of $u(t, x;\omega)$ (solid green) and $v(t,x;\omega)$ (dashed red) for $\omega = 0.7$, see \eqref{eqn:u-shock} and \eqref{eqn:v-shock}, respectively. {\bf Left:} $t=0$. {\bf Middle:} $t_1 (0.7)=5/7$. Here $u(t_1(\omega),x;0.7)$ developed two shocks, whereas $v(t_1(\omega),x;0.7)$ developed only one. {\bf Right:} $t_2(0.75) = 10/7$. Here, the two shocks of $u$ have merged, whereas $v$ still has only one shock.}
    \label{fig:u0v0}
\end{figure}

What do typical solutions $u$ or $v$ look like? $u_0(x;\omega)$ and $v_0 (x;\omega)$ both consist of two consecutive step-like structures that develop a shock at finite time, see Fig.\ \ref{fig:u0v0}, left panel. By explicit calculation (using the method of characteristics), the left step develops into a shock for both $u$ and $v$ at time $t_1 (\omega) =(2\omega)^{-1}\in [1/2,\infty)$. For $u$, the right step also develops into a shock at $t_1(\omega)$. For  $v$, however, the shock develops at $t'_1(\omega)= (2(1-\omega))^{-1}$ (see Fig.\ \ref{fig:u0v0}, middle panel). Once formed, the right shocks of $u$ and $v$ travel (left) at speeds of $-\omega/2$ and $-(1-\omega)/2$, respectively, whereas the left shocks travel (right) at a speed of $\omega/2$. Therefore, At $t_2 (\omega)=\omega^{-1}\in [1,\infty)$ the two shocks of $u(t,x;\omega)$ merge at $x=0$, where they remain for all $t\geq t_2 (\omega)$, i.e., with probability $1$. For $v$, however, the shocks merge at different times and will travel to the right at speed $1$.

Notwithstanding the differences in the dynamics, the initial CDF corresponding to both $u_0$ and $v_0$, denoted for simplicity by $F_0(x,v)$, is given by
\begin{itemize}
    \item for $x \in [-0.25, 0.25]$:
    \begin{align*}
    F_0(x,U) = \left\{ \begin{array}{cc} 
     0,& U<0, \\ 1, & U>0;
    \end{array}\right.
    \end{align*}
   \item for $x >0.25$:
   \begin{align*}
    F_0(x,U) = \left\{ \begin{array}{cc} 
     1, & U \geq 0, \\ 1- \frac{U}{g_1(x-\half)}, & 0 \geq U \geq g_1(x-\half),\\
     0, & U< g_1(x-\half); 
    \end{array}\right.
    \end{align*}
    \item for $x <-0.25 $:
    \begin{align*}
    F_0(x,U) = \left\{ \begin{array}{cc} 
     0, & U<0, \\ \frac{U}{g_2(x+\half)}, & 0\leq U \leq g_2(x+\half),\\
     1, & U \geq g_2(x+\half).
    \end{array}\right.
    \end{align*}
\end{itemize}

Again, we compare the directly computed (using Monte Carlo)  CDF and PDF corresponding to $u$ and $v$ with the directly evolved versions according to \eqref{eqn:CDF_noshock} and \eqref{eqn:PDF_PDE_nl}, respectively. In the following, we compare the solutions at three distinct time points relative to the times of shock formation, as shown in Fig.~\ref{fig:u0v0} for illustration. In all these simulations, the Monte Carlo simulations use $2\times 10^4$ samples, and the discretization parameters used are $\Delta x = 0.01$, $\Delta t = 0.0025$, and $\Delta U = 0.01$. 

Figs.\ \ref{fig:shock1} and \ref{fig:shock2} present the results at $t=0.4$, prior to the formation of any shocks, i.e., $t< \min_{\omega}t_1(\omega) = t_1(1)=0.5$. Since no shocks have formed, the three solutions agree as expected (Sec.\ \ref{sec:nonlin}).

At $t=0.8$ (see Figs.\ \ref{fig:t8_1}--\ref{fig:t8_4}), $u$ has developed two shocks for all $\omega >5/8$ and none otherwise. On the other hand, $v$ developed exactly one shock, whose location depends on the value of $\omega$. Here, the analysis of Sec.\ \ref{sec:nonlin} is no longer valid, and indeed, as Remark \ref{rem:noShock} implies, the directly-computed PDF $f(t,x;U)$ does not agree with those of $u$ and $v$. What may be surprising is that the directly-computed (via Monte Carlo simulations) pointwise statistics of $u$ and $v$ agree. This is due to a symmetry between $u$ and $v$: for $x<0$, they behave exactly the same at $\min_{\omega}t_1 (\omega)<t<\min_{\omega}t_2 (\omega)$. Furthermore, while they differ for $x>0$, there is a one-to-one correspondence between $u(t,x;\omega)$ and $v(t,x;1-\omega)$ for these $x$ values, and therefore the pointwise statistics agree.

Finally, we consider $t = 2>t_2(\omega)$ for $\omega >1/2$, when the two shocks have collided and merged for $u$ (see Fig.\ \ref{fig:u0v0}, right panel). As can be seen in Figs.~\ref{fig:shock_t2_1} and \ref{fig:shock_t2_2}, the pointwise statistics of $u$ and $v$ already disagree here (since the shocks have not merged for $v$, or indeed might not have formed yet).


\begin{figure}[!h]
\includegraphics[width=1.0\textwidth]{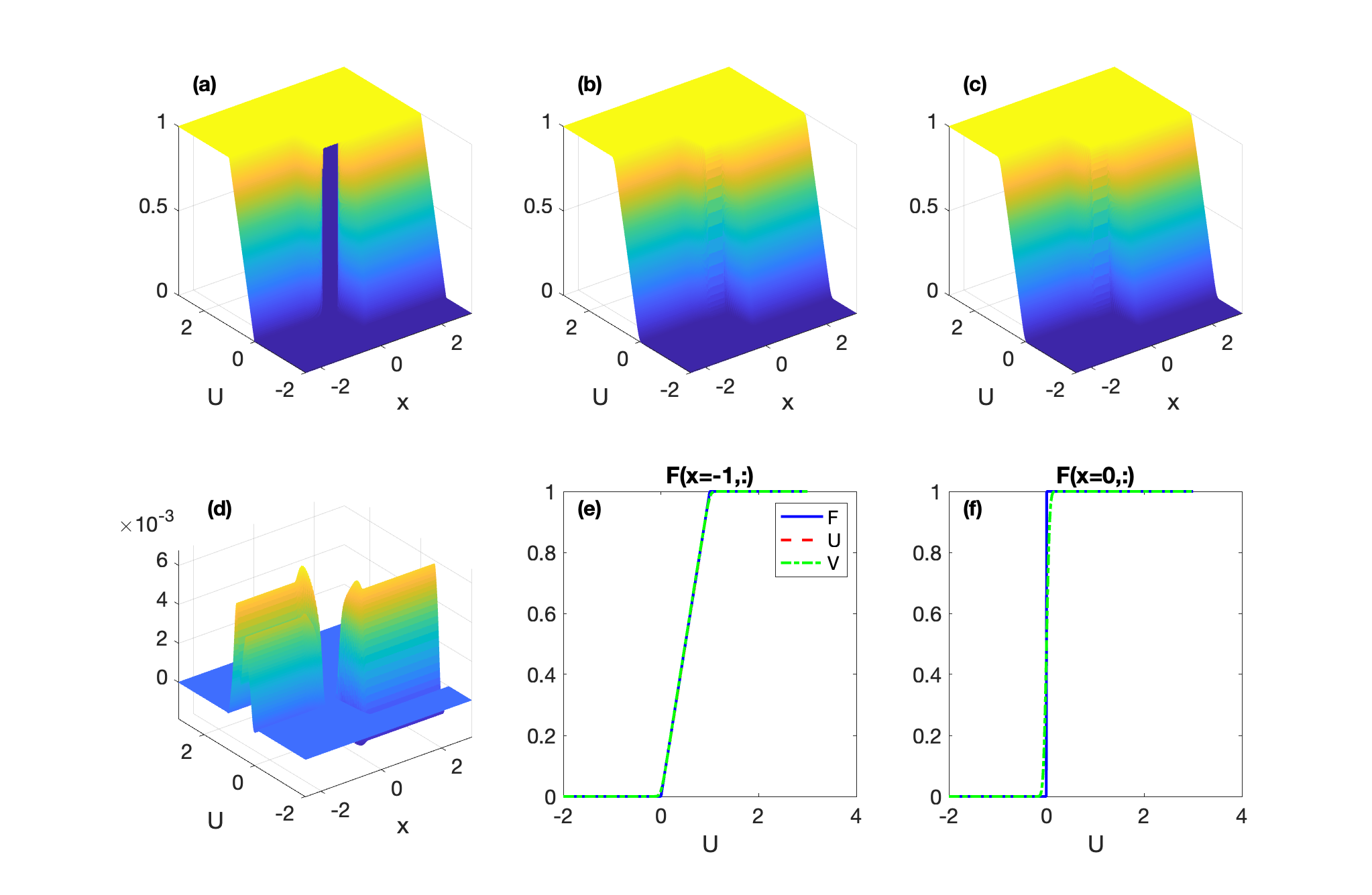}
\caption{Plots (a)–(c) show the results of solving \eqref{eqn:CDF_noshock} for $F$ and of performing direct Monte Carlo simulations of \eqref{eqn:u-rare} and \eqref{eqn:v-rare}, respectively. Plot (d) displays the difference between the solutions of \eqref{eqn:u-rare} and \eqref{eqn:v-rare}. Although the 3D plot may visually suggest a significant discrepancy (based on color) between plots (a) and (b) or (c), the actual $L^1$ error between (a) and (b) is only 0.0468. To further illustrate this point, plots (e) and (f) show the CDFs at slices $x = -1$ and $x = 0$, respectively, with the direct solution of $F$ (solid blue), $u$ (dashed red), and $v$ (dashed green).}
\label{fig:shock1}
\end{figure}

\begin{figure}[!h]
\includegraphics[width= 0.8\textwidth, height= 0.3\textwidth]{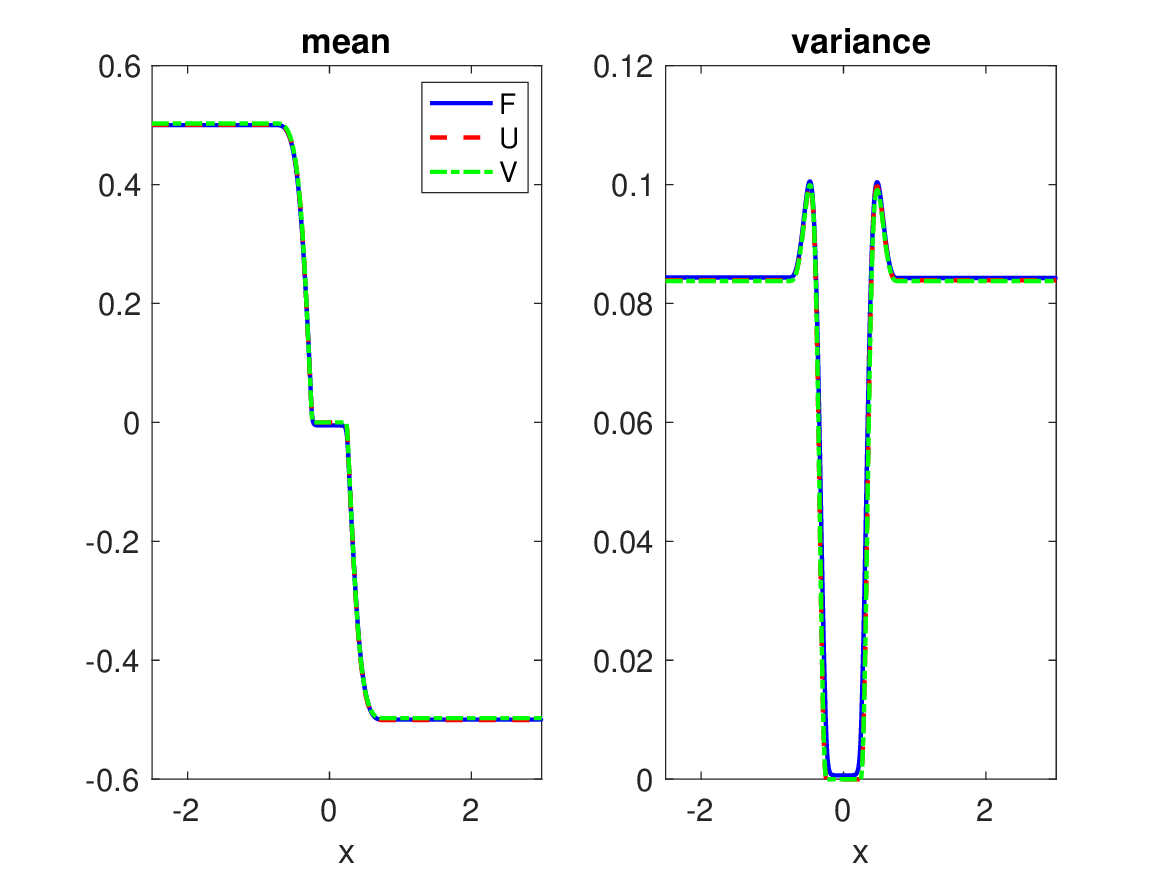}
\caption{Same settings as Fig.\ \ref{fig:shock1},  comparison of mean (left) and variance (right):  $u$ \eqref{eqn:u-shock} (dashed red), $v$ \eqref{eqn:v-shock} (dashed green) and the direct evolution of the \revO{CDF via \eqref{eqn:CDF_noshock} (solid blue)}.}
\label{fig:shock2}
\end{figure}

\begin{figure}[!h]
\includegraphics[width=0.9\textwidth, height=0.3\textwidth]{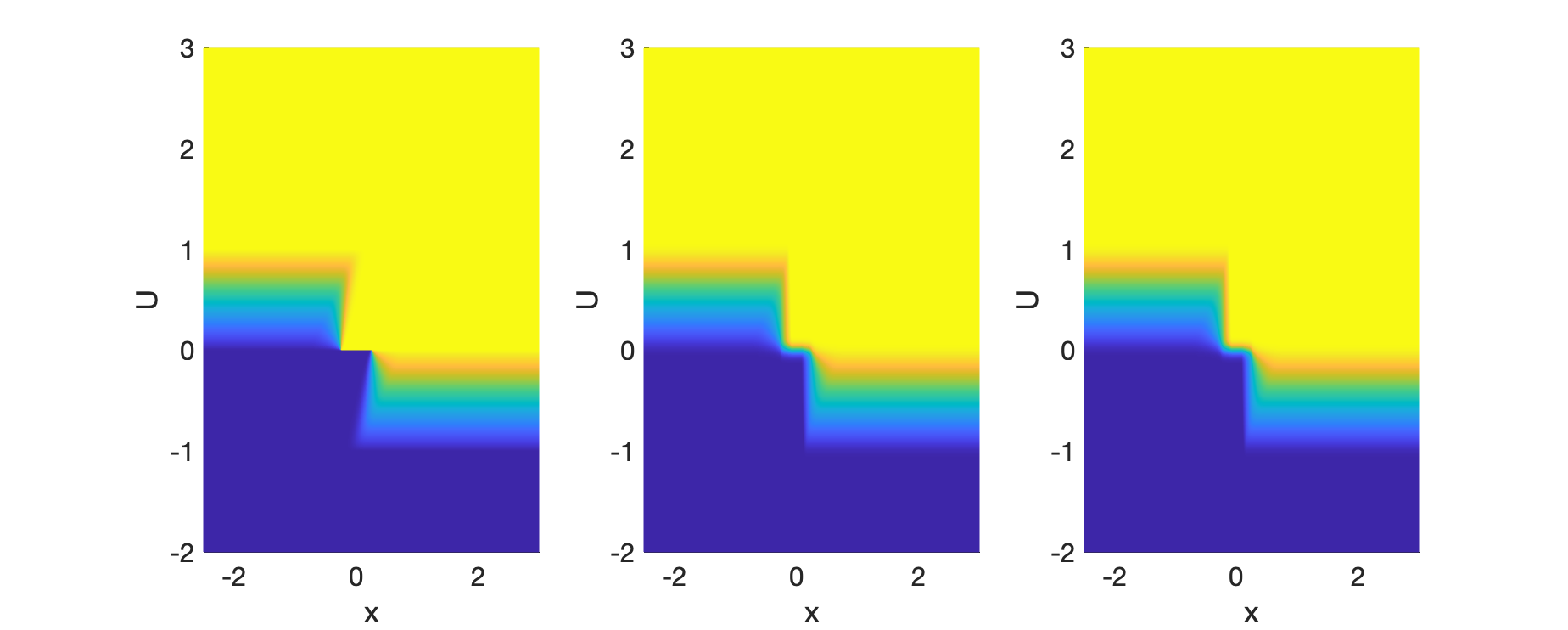}
\caption{Settings of Sec.\ \ref{sec:ShockExample} at $t = 0.8$. Comparison of the CDFs obtained from the directly evolved CDF \eqref{eqn:CDF_noshock} (left), $u$  \eqref{eqn:u-shock} (middle) and $v$  \eqref{eqn:v-shock} (right).}
\label{fig:t8_1}
\end{figure}

\begin{figure}[!h]
\includegraphics[width= 0.85\textwidth, height= 0.3\textwidth]{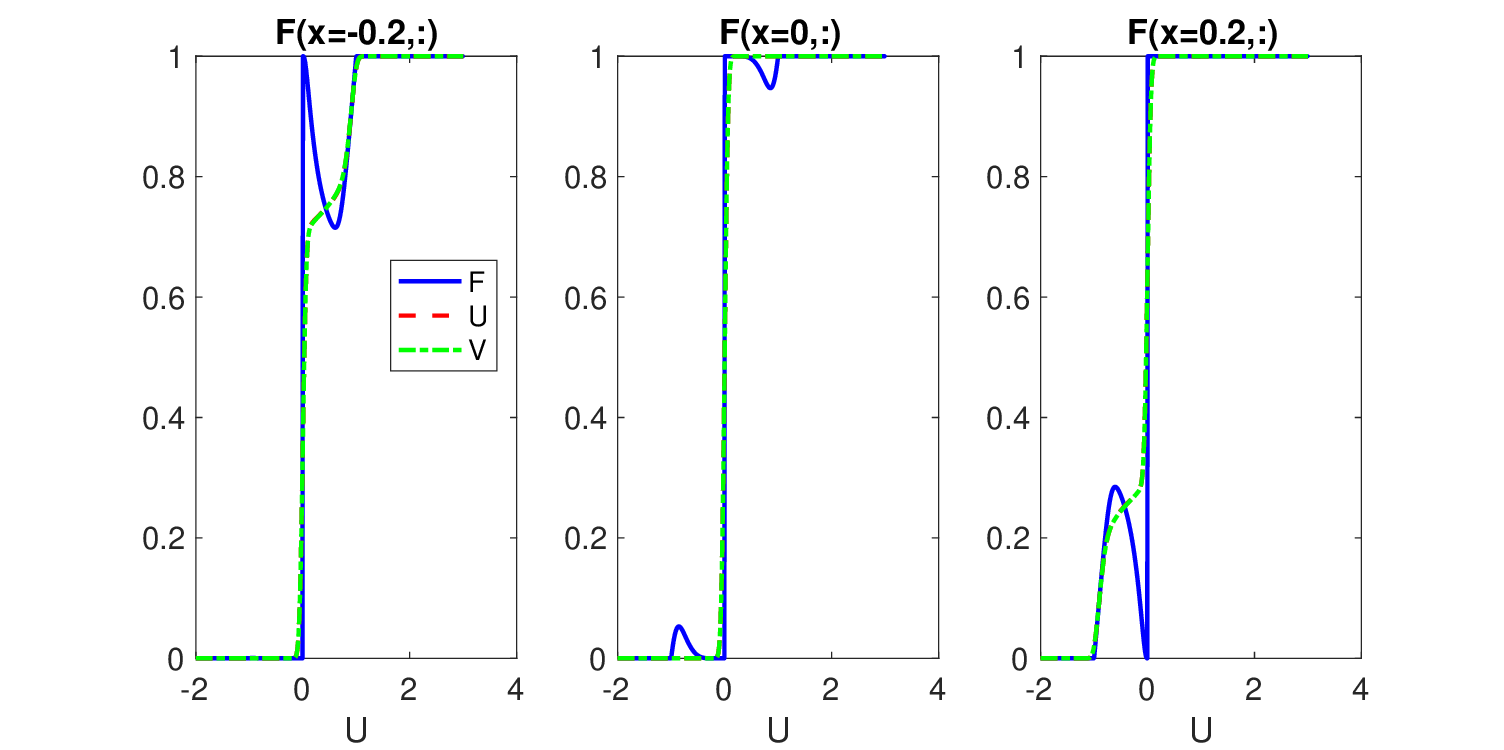}
\caption{Settings of Sec.\ \ref{sec:ShockExample} at $t = 0.8$. Cross-sections of Fig.~\ref{fig:t8_1} at three different locations $x=-0.2$, $x=0$ and $x=0.2$, with direct solution of $F$ (solid blue), $u$ (dashed red), and $v$ (dashed green).}
\label{fig:t8_3}
\end{figure}

\begin{figure}[!h]
\includegraphics[width= 0.7\textwidth, height= 0.35\textwidth]{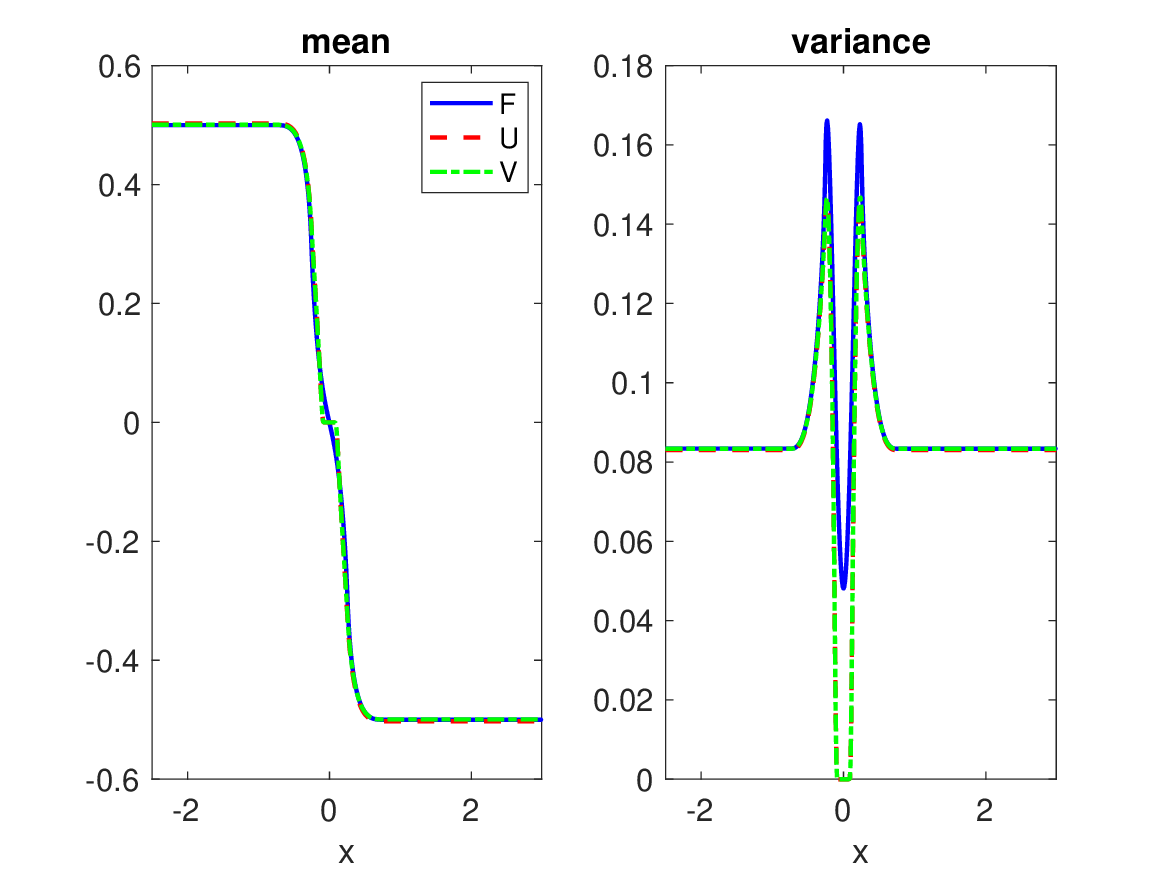}
\caption{Same settings and legend as Fig.\ \ref{fig:t8_3}. Comparison of mean and variance.}
\label{fig:t8_4}
\end{figure}

\begin{figure}[!h]
\includegraphics[width=\textwidth, height=0.3\textwidth]{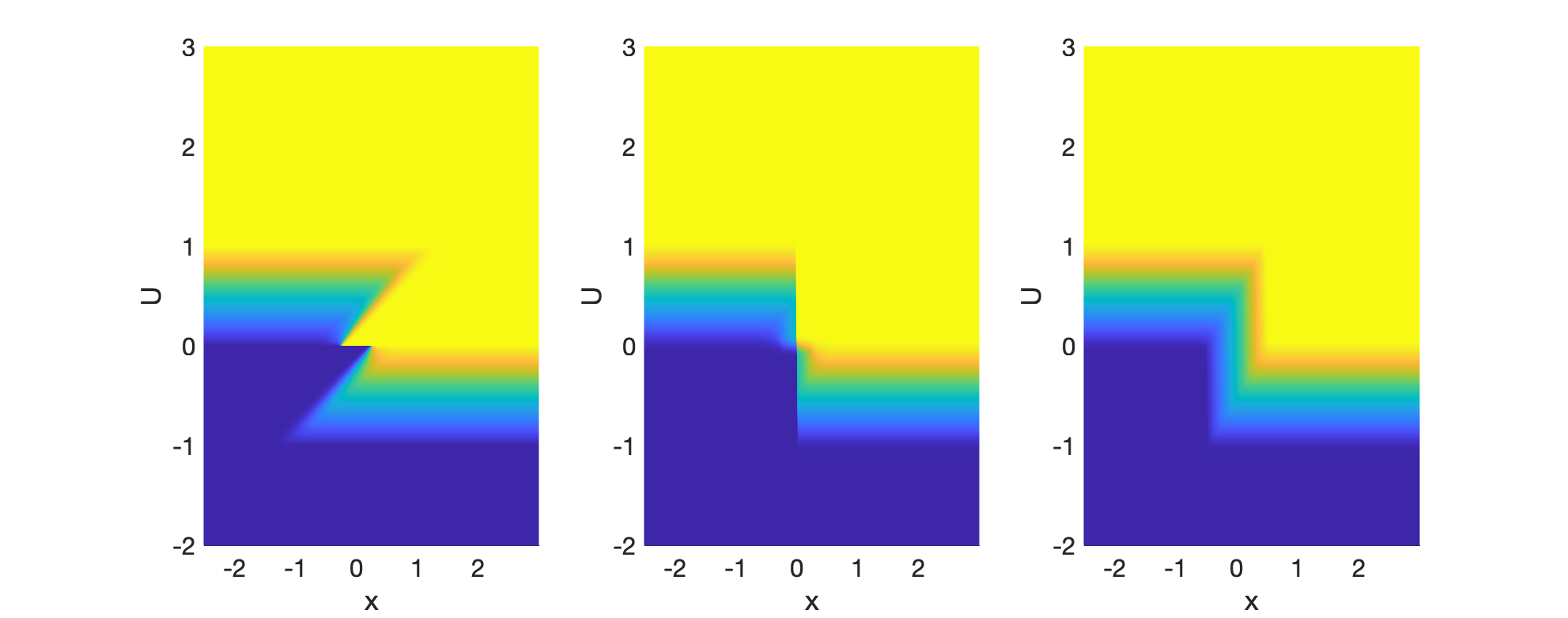}
\caption{Settings of Sec.\ \ref{sec:ShockExample} at $t = 2.0$. Comparison of CDF obtained from new model \eqref{eqn:CDF_noshock} (left), $u$  \eqref{eqn:u-shock} (middle), and $v$  \eqref{eqn:v-shock} (right).}
\label{fig:shock_t2_1}
\end{figure}

\begin{figure}[!h]
\includegraphics[width=\textwidth, height=0.3\textwidth]{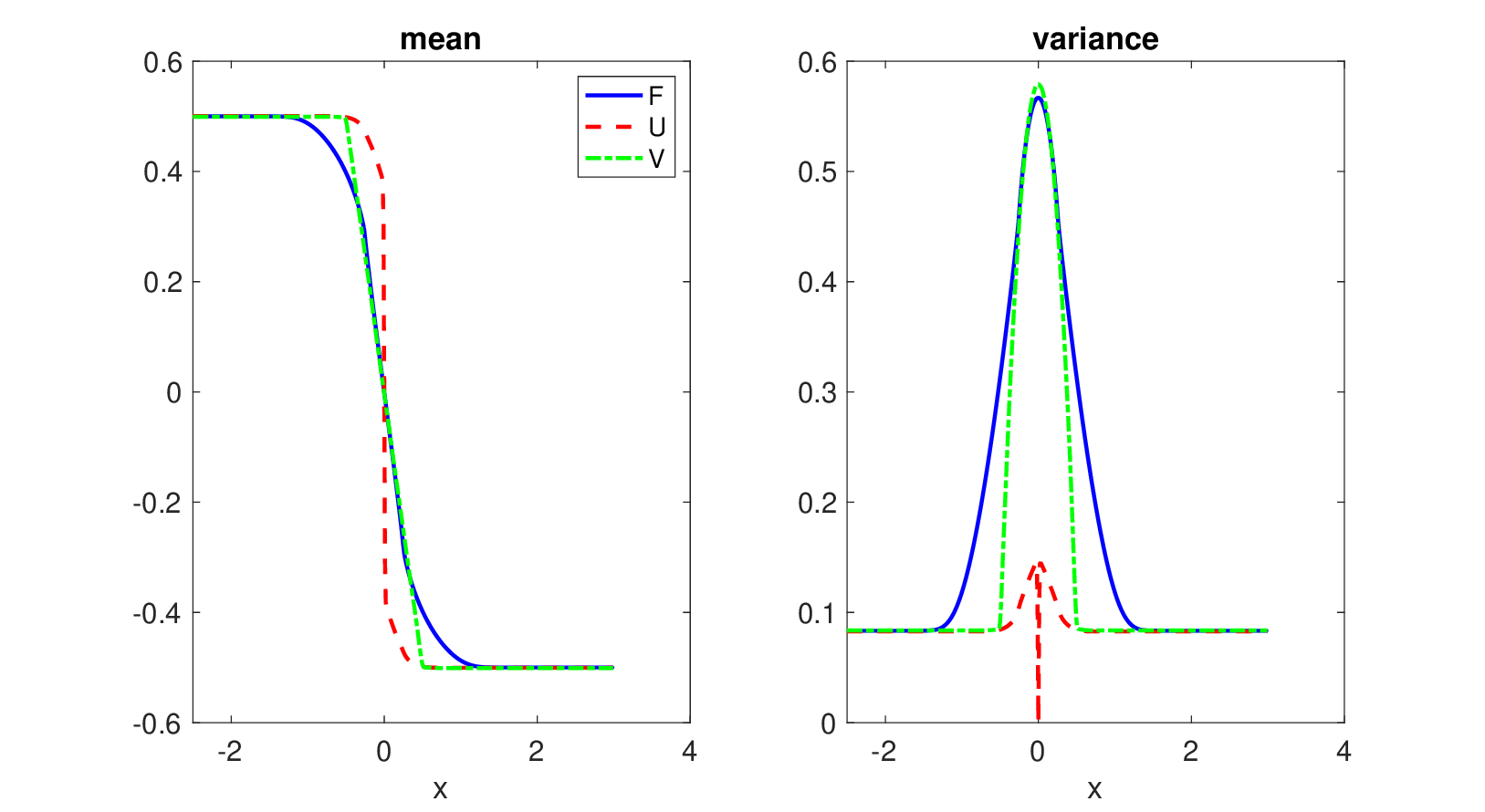}
\caption{Settings of Sec.\ \ref{sec:ShockExample} at $t = 2.0$. Comparison of mean and variance for the directly solved $F$ (solid blue), $u$ (dashed red), and $v$ (dashed green).}
\label{fig:shock_t2_2}
\end{figure}

\bigskip
\paragraph{\bf Acknowledgement:} L.\ Wang would like to thank Professors \ Shi Jin and Daniel Tartakovsky for fruitful discussions. The work 
of A.\ Chertock was supported in part by NSF grant DMS-2208438. The work of A.\ Sagiv was supported in part by the NSF grant DMS-2508811 and Binational Science 
Foundation Research Grant \#2022254. The work of L. Wang is partially supported by NSF grant DMS-1846854, DMS-2513336, and the Simons Fellowship.

\appendix
\section{Numerical scheme for the nonlocal equation \eqref{eqn:PDF_PDE_nl}} \label{sec:appendix}
First, reformulate \eqref{eqn:PDF_PDE_nl} into a conservative form in $x$ and then adopt the central upwind scheme developed in \cite{KNP01}. More precisely, rewrite \eqref{eqn:PDF_PDE_nl} into 
\begin{equation} \label{eqn:f00}
\partial_t f  + \partial_x \left[  a''(U) \int_{-\infty}^U f(t,x,v) \rd v + a'(U) f(U)   \right] = 0\,.
\end{equation} 
Let 
\[
F(f) = \frac{\partial}{\partial U} \left[  a'(U) \int_{-\infty}^U f(t,x,v) \rd v \right] \,,
\]
denote the flux, then \eqref{eqn:f00} reduces to 
\begin{equation} \label{eqn:f001}
\partial_t f + \partial_x F(f) =0\,,
\end{equation}
which can be considered as a conservation law for $f$. In this regard, many shock-capturing schemes can be employed. We select the central upwind scheme for its simplicity and high-order accuracy. 

Choose the computational domain $x_L \leq x_R$, $U_L \leq U \leq U_R$, and $t\leq 0$ and denote
\[
f_{jk}^n = f(t_n, x_j, U_k), \quad 0\leq n \leq N_t, ~ 0\leq j \leq N_x,  ~ 0\leq k \leq N_U\,\,,
\]
where $t_n = n\Delta t$, $x_j = x_L + j \Delta x $, $U_k = U_L + k \Delta U$. Here $\Delta t$, $\Delta x$ and $\Delta U$ are time step, mesh size in $x$ and $U$, respectively.
\begin{remark*}
  \revO{We note that \eqref{eqn:PDF_PDE_nl} itself does not require any a priori
knowledge of the range of $U$, since the integral is defined over $
(-\infty,U)$ for every $U \in \RR$. In numerical computations, however,
one must choose a finite computational domain $[U_L,U_R]$. This can be
consistently estimated from the range of the initial data: values of $U$ which
are not attained at $t=0$ (or attained with probability zero) will not be attaind for
all times.}
\end{remark*}
Then, \eqref{eqn:f001} can be discretized as 
\begin{equation} \label{sch:centralupwind}
\frac{f_{jk}^{n+1} - f_{jk}^n }{\Delta t} + \frac{ H_{j+1/2, k} - H_{j-1/2,k} }{\Delta x} =0\,,
\end{equation}
where
\[
H_{j+1/2,k} = \frac{a_\jhalf ^+ F(f^-_{\jhalf,\revO{k}}) - a_\jhalf^- F(f^+_{\jhalf, \revO{k}})}{a_\jhalf^+ - a_\jhalf^-} 
+ \frac{a_\jhalf^+ a_\jhalf^-}{a_\jhalf^+ - a_\jhalf^-} \left( f^+_{\jhalf, \revO{k}} - f^-_{\jhalf, \revO{k}} \right)\,,
\]
and 
\[
a_\jhalf^+ = \max \left\{   \left( \frac{\partial F}{\partial f} \right)_{\jhalf, \revO{k}}^n, 0 \right\}, \quad 
a_\jhalf^- = \min \left\{   \left( \frac{\partial F}{\partial f} \right)_{\jhalf, \revO{k}}^n, 0 \right\}\,.
\]
Here $f_\jhalf^\pm$ are constructed using the minmod slope limiter. That is,
\[
f_{\jhalf, \revO{k}}^- = f_{j,\revO{k}} + \frac{\Delta x}{2} (f_x)_{j, \revO{k}}, \qquad f_{\jhalf, \revO{k}}^+ = f_{j+1, \revO{k}} - \frac{\Delta x}{2} (f_x)_{j+1, \revO{k}}\,,
\]
with
\[
(f_x)_{j, \revO{k}} = \text{minmod} \left(  \theta \frac{f_{j, \revO{k}} - f_{j-1,\revO{k}}}{\Delta x}, ~ \frac{f_{j+1, \revO{k}}-f_{j-1, \revO{k}}}{2\Delta x} , ~ \theta \frac{f_{j+1, \revO{k}} - f_{j, \revO{k}}}{\Delta x}  \right)\,,
\]
and
\begin{equation*}
\text{minmod}(z_1, z_2, \cdots ) = \left\{ \begin{array}{cc} \min_j \{ z_j\}  & z_j >0 \quad \forall j\,,
\\ \max_j \{z_j\}  & z_j < 0 \quad \forall j \,, \\ 0 & \text{otherwise} \,. \end{array} \right. 
\end{equation*}
Note that without slope limiter, i.e., $(f_x)_{j, \revO{k}} \equiv 0$ for any $j$, and $a_\jhalf$ keeps its sign, then the scheme \eqref{sch:centralupwind} reduces to the upwind scheme.

\bibliography{UQ_new_ref}
\bibliographystyle{siam}

\end{document}